\documentclass[11pt]{article}
\usepackage{latexsym,amsfonts,amssymb,amsmath,amsthm}
\usepackage{graphicx}

\usepackage[usenames,dvipsnames]{color}
\usepackage{ulem}

\begin{document}
\setlength{\baselineskip}{16pt}

\parindent 0.5cm
\evensidemargin 0cm \oddsidemargin 0cm \topmargin 0cm \textheight 22cm \textwidth 16cm \footskip 2cm \headsep
0cm

\newtheorem{theorem}{Theorem}[section]
\newtheorem{lemma}{Lemma}[section]
\newtheorem{proposition}{Proposition}[section]
\newtheorem{definition}{Definition}[section]
\newtheorem{example}{Example}[section]
\newtheorem{corollary}{Corollary}[section]

\newtheorem{remark}{Remark}[section]

\numberwithin{equation}{section}

\def\p{\partial}
\def\I{\textit}
\def\R{\mathbb R}
\def\C{\mathbb C}
\def\u{\underline}
\def\l{\lambda}
\def\a{\alpha}
\def\O{\Omega}
\def\e{\epsilon}
\def\ls{\lambda^*}
\def\D{\displaystyle}
\def\wyx{ \frac{w(y,t)}{w(x,t)}}
\def\imp{\Rightarrow}
\def\tE{\tilde E}
\def\tX{\tilde X}
\def\tH{\tilde H}
\def\tu{\tilde u}
\def\d{\mathcal D}
\def\aa{\mathcal A}
\def\DH{\mathcal D(\tH)}
\def\bE{\bar E}
\def\bH{\bar H}
\def\M{\mathcal M}
\renewcommand{\labelenumi}{(\arabic{enumi})}

\def\disp{\displaystyle}
\def\undertex#1{$\underline{\hbox{#1}}$}
\def\card{\mathop{\hbox{card}}}
\def\sgn{\mathop{\hbox{sgn}}}
\def\exp{\mathop{\hbox{exp}}}
\def\OFP{(\Omega,{\cal F},\PP)}
\newcommand\JM{Mierczy\'nski}
\newcommand\RR{\ensuremath{\mathbb{R}}}
\newcommand\CC{\ensuremath{\mathbb{C}}}
\newcommand\QQ{\ensuremath{\mathbb{Q}}}
\newcommand\ZZ{\ensuremath{\mathbb{Z}}}
\newcommand\NN{\ensuremath{\mathbb{N}}}
\newcommand\PP{\ensuremath{\mathbb{P}}}
\newcommand\abs[1]{\ensuremath{\lvert#1\rvert}}

\newcommand\normf[1]{\ensuremath{\lVert#1\rVert_{f}}}
\newcommand\normfRb[1]{\ensuremath{\lVert#1\rVert_{f,R_b}}}
\newcommand\normfRbone[1]{\ensuremath{\lVert#1\rVert_{f, R_{b_1}}}}
\newcommand\normfRbtwo[1]{\ensuremath{\lVert#1\rVert_{f,R_{b_2}}}}
\newcommand\normtwo[1]{\ensuremath{\lVert#1\rVert_{2}}}
\newcommand\norminfty[1]{\ensuremath{\lVert#1\rVert_{\infty}}}
\newcommand{\ds}{\displaystyle}

\title{Asymptotic dynamics in a two-species chemotaxis model with non-local terms }
\author{
Tahir  Bachar Issa and Rachidi B.  Salako   \\
Department of Mathematics and Statistics\\
Auburn University\\
Auburn University, AL 36849\\
U.S.A. }

\date{}
\maketitle

\begin{abstract}
In this study, we  consider the following extended attraction two species chemotaxis system of  parabolic-parabolic-elliptic type with nonlocal terms
 \[
\begin{cases}
u_t=d_1\Delta u-\chi_1\nabla (u\cdot \nabla w)+u\left(a_0-a_1u-a_2v-a_3\int_{\Omega}u-a_4\int_{\Omega}v\right),\quad x\in \Omega \quad\cr
v_t=d_2\Delta v-\chi_2\nabla (v\cdot \nabla w)+v\left(b_0-b_1u-b_2v-b_3\int_{\Omega}u-b_4\int_{\Omega}v\right),\quad x\in \Omega \quad\cr
0=d_3\Delta w+k u+lv-\lambda w,\quad x\in \Omega \quad\cr

\end{cases}
\]
under homogeneous Neuman boundary conditions in a bounded domain $\Omega \subset \mathbb{R}^n(n\ge1)$ with smooth boundary, where  $a_0,b_0,\, \,a_1,k,l,\lambda,\chi_i,d_i$ and $ b_2$ are positive and $a_2,\, a_3, \, a_4, \, b_1,\, b_3,$ and $b_4$ are real numbers. We first prove the global existence of non-negative classical solutions for various explicit parameter regions. Next,  under some  further explicit conditions on the coefficients $a_i,\, b_i,d_i,l,k,\lambda$ and on the chemotaxis sensitivities  $\chi_i$, we show  that    the system has a unique positive constant steady state solution which is globally asymptotically stable. Finally, we also find some  explicit conditions on the coefficients $a_i,\, b_i,d_i,l,k,\lambda$ and on the chemotaxis sensitivities  $\chi_i$  for which the phenomenon of competitive exclusion occurs  in the sense that  as time goes to infinity, one of the species dies out and the other reaches its carrying capacity . The method of eventual comparison is used to study the asymptotic behavior.

\end{abstract}

  \medskip


  \medskip


\section{Introduction and Statement of the Main Results}
\label{S:intro}

 Bacteria chemotaxis, or simply chemotaxis is the directed movement of biological cells or micro organisms in response to chemical signals in their environment. Bateria chemotaxis  is crucial for many aspects of behaviour, including the location of food sources, avoidance of predators and attracting mates, slime moud aggregation, tumour angiogenesis, and primitive steak formation (see \cite{ISM04,  KJPJAS03} and the references therein).  Recent studies, \cite{DAL1991}, suggest  that chemotaxis plays also a crucial role in  macroscopic process such as population dynamics , gravitational collapse, etc.  Indeed, M. J. Kennedy and J. G. Lawless conclude in \cite{MJKJGL85} `` Thus, chemotaxis may be one mechanism by which denitrifiers successfully compete for available $N0_3^{-}$and $N0_2^{-}$, and which may facilitate the survival of naturally occurring populations of some denitrifiers. ''  and D. A. Lauffenburger in \cite{DAL1991} mentioned `` Current results indicate that cell motility and chemotaxis properties can be as important to population dynamics as cell growth kinetic properties, so that greater attention to this aspect of microbial behavior is warranted in future studies of microbial ecology. ''

In  the 1970s, Keller and Segel proposed a celebrated mathematical model to describe the aggregation process of Dictyostelium discoideum, a soil-living amoebea \cite{KS1970, KS71}. Following the pioneering works of Keller and Segel, chemotaxis models have attracted the attention of many researchers  in mathematical biology.  It is well known that chemotactic-cross diffusion has a very strong destabilizing action in space dimension $n\geq 2$ in the sense that finite-time blow-up of some classical solutions may occurs (see \cite{CoEsVi11, HVJ1997a, JaW92, W2011b} for one species chemotaixis model and \cite{ASV2009} two species chemotaxis models ). However, it is  also known that logistic sources of Lotka-Volterra type preclude such blow-up phenomenon (see \cite{TW07, ITBWS16, RBSWS17a} for one species and \cite{TW12,NT13} for two species) and that, at least numerically, chemotaxis may
exhibit quite a rich variety of colorful dynamical features, up to periodic and even chaotic solution behavior \cite{kuto_PHYSD, PaHi}.

In this work, we study the long-term behaviour of the following extended cooperative-competitive attraction two species chemotaxis system of parabolic-parabolic-elliptic type with   nonlocal terms
\begin{equation}
\label{u-v-w-eq1}
\begin{cases}
u_t=d_1\Delta u-\chi_1\nabla (u\cdot \nabla w)+u\left(a_0-a_1u-a_2v-a_3\int_{\Omega}u-a_4\int_{\Omega}v\right),\quad x\in \Omega \quad\cr
v_t=d_2\Delta v-\chi_2\nabla (v\cdot \nabla w)+v\left(b_0-b_1u-b_2v-b_3\int_{\Omega}u-b_4\int_{\Omega}v\right),\quad x\in \Omega \quad\cr
0=d_3\Delta w+k u+lv-\lambda w,\quad x\in \Omega \quad\cr
\frac{\p u}{\p n}=\frac{\p v}{\p n}=\frac{\p w}{\p n}=0, \quad x\in\p \Omega,
\end{cases}
\end{equation}
where $\Omega$ is a bounded subset of $\mathbb{R}^n$ with smooth
boundary,  $u(x,t)$ and $v(x,t)$ are the population densities of two
species attracted by  the same chemoattractant substance with density  $w(x,t)$;  $\chi_{i}>0, i=1,2,$ are the constant chemotactic sensitivities; $d_i>0,i=1,2,3$, are diffusion coeficients; $k,l$ and $\lambda$ are positive and represent respectively, the creation and degradation rate of the chemical substance; $a_0,b_0>0$, describe the intrinsic growth rate of the species $u$ and $v$ respectively; $a_1,b_1>0$, describe the self limitation effect of the species $u$ and $v$ respectively; $b_1 \in \mathbb{R}$ (resp. $a_2 \in \mathbb{R}$ ) describe the local effect of the species $u$ (resp. of the species $v$) on the species $v$ (resp. on the species $u$) and the nonlocal term $\int_{\Omega}u$ (resp. $\int_{\Omega}v$) describe the effect of the total mass of $u$ (resp. of $v$) on the growth of the two species; $a_3, \,a_4,,\, b_3,$  $b_4$ are real numbers (see \cite{NT13, ITBWS16} for more details on this model).

We now review briefly the existing works on various special cases of system  \eqref{u-v-w-eq1} and motivate our current study of the asymptotic dynamics  of \eqref{u-v-w-eq1}. If  $d_1=d_2=d_3=1$ and  $a_i, b_i >0$ ($i=1,2$), Negreanu and Tello \cite{NT13}  proved that  system \eqref{u-v-w-eq1} has a unique globally stable homogeneous steady state $(u^*,v^*,w^*)$ where
 $$
u^*=\frac{a_0(b_2+|\Omega|b_4)-b_0(a_2+|\Omega|a_4)}{(b_2+|\Omega|b_4)(a_1+|\Omega|a_3)-(a_2+|\Omega|a_4)(b_1+|\Omega|b_3)},$$
$$ v^*=\frac{a_0(b_1+|\Omega|b_3)-b_0(a_1+|\Omega|a_3)}{(b_1+|\Omega|b_3)(a_2+|\Omega|a_4)-(a_1+|\Omega|a_3)(b_2+|\Omega|b_4)},
$$
 and
$$ w^*=\frac{ku^*+lv^*}{\lambda},
$$
 under the assumption
\begin{equation}
\label{eq-int-000}
a_1> 2k(\chi_1+\chi_2)+ b_1+|b_3|+|a_3| \quad \text{and} \quad b_2> 2l(\chi_1+\chi_2)+ a_2+|a_4|+|b_4|.
\end{equation}

System \eqref{u-v-w-eq1} without nonlocal terms  ($a_3=a_4=b_3=b_4=0$)  and  with $a_0=a_1=\mu_1$, $a_2=\mu_1\tilde a_2$, $b_0=b_2=\mu_2$, $b_1=\mu_2\tilde b_1$,  becomes
\begin{equation}
\label{u-v-w-eq1bis}
\begin{cases}
u_t=d_1\Delta u-\chi_1\nabla (u\cdot \nabla w)+\mu_1u\left(1-u-\tilde a_2v\right),\quad x\in \Omega \quad\cr
v_t=d_2\Delta v-\chi_2\nabla (v\cdot \nabla w)+\mu_2v\left(1-\tilde b_1u-v\right),\quad x\in \Omega \quad\cr
0=d_3\Delta w+k u+lv-\lambda w,\quad x\in \Omega \quad\cr
\frac{\p u}{\p n}=\frac{\p v}{\p n}=\frac{\p w}{\p n}=0, \quad x\in\p \Omega.
\end{cases}
\end{equation}

If $d_3=1$ and  $\tilde a_2,\tilde b_2 \in [0,1)$, Tello and Winkler \cite{TW12} show that 
$(u^*,v^*,w^*)$ is a unique globally stable steady state  for \eqref{u-v-w-eq1bis} under the assumption
\begin{equation}
\label{eq-int-001}
\mu_1> 2(\chi_1+\chi_2)+\tilde a_2\mu_2 \quad \text{and} \quad  \mu_2> 2(\chi_1+\chi_2)+\tilde b_1\mu_1.
\end{equation}
Note that the assumption \eqref{eq-int-000}(resp. \eqref{eq-int-001}) is not natural in the sense when $\chi_1=\chi_2=0$, \eqref{eq-int-000} (resp. \eqref{eq-int-001}) does not hold trivially. Recently,  Black, Lankeit and Mizukam in  \cite{TBJLMM16}, used the powerful tool of eventual comparison method as called in  \cite{TBJLMM16}  and obtained when $\tilde a_2,\tilde b_2 \in [0,1)$, the global asymptotic stability of 
 $(u^*,v^*,w^*)$ for system \eqref{u-v-w-eq1bis} under  the  natural conditions
\begin{equation}
\label{eq-int-002}
q_1 \in [0, \frac{d_3}{k})\cap[0,\frac{\tilde a_2d_3}{l}), q_2 \in [0, \frac{d_3}{2l})\cap[0,\frac{\tilde  b_1d_3}{k})
\end{equation}
\begin{equation}
\label{eq-int-003}
\tilde  a_2 \tilde b_1d_3^2< \left(d_3-2kq_1\right)\left(d_3-2lq_2\right),
\end{equation}
where $q_1=\frac{\chi_1}{\mu_1}$ and $q_2=\frac{\chi_2}{\mu_2}.$

The goal of our current study can be summarized in two main points. First, we extend the results by  Black, Lankeit and Mizukam in \cite{TBJLMM16} to the case with nonlocal terms of system $\eqref{u-v-w-eq1}$ and show the efficiency of the method of eventual comparison even in the case of non local terms. Secondly, motivated by the results in \cite{STW13}, we prove  the phenomenon of competitive exclusion for system ~\eqref{u-v-w-eq1} under some natural conditions on the parameters. In \cite{STW13}, the authors proved by the eventual comparison method, the phenomenon of competitive exclusion for system ~\eqref{u-v-w-eq1bis} under the assumptions
$$k\geq 0,\, q_1\geq 0,\,q_2\geq 0,\,q_1\leq \tilde  a_2,\,q_2<\frac{1}{2}$$
and
\begin{equation}
\label{eq-int-004}
kq_1+\max\Big\{q_2, \frac{\tilde  a_2-\tilde a_2q_2}{1-2q_2},  \frac{kq_2-\tilde  a_2q_2}{1-2q_2}\Big\} <1.
\end{equation}

Throughout the paper, we use the following standard notations:
$$
(a)_-=\max\{0, -a\},\quad (a)_+=\max\{0, a\} \quad \forall a\in \R.
$$
Let
$$
C^0(\overline{\Omega})=\{u\ : \ \Omega\rightarrow \R , \ \text{bounded and uniformly continuous}\}
$$
and for every $u\in C^0(\overline{\Omega})$ we define
$$
\|u\|_{\infty}=\|u\|_{C^0(\overline{\Omega})}=\sup_{x\in\Omega}|u(x)|.
$$

  For convenience, we introduce the following assumptions.

\medskip
\noindent {\bf (H1)} {
\begin{equation}
\begin{cases}
\label{global-existence-cond-eq}
a_1>(b_1)_-+|\Omega|\left((a_3)_-+(b_3)_-\right)+k\frac{\chi_1+\chi_2}{d_3}\cr
 b_2> (a_2)_-+|\Omega|\left((a_4)_-+(b_4)_-\right)+l\frac{ \chi_1+ \chi_2}{d_3}.
\end{cases}
 \end{equation}
}

\medskip
\noindent {\bf (H2)} {
\begin{equation}
\begin{cases}
\label{L1-boundedness-cond-eq-localcompetition}
a_1>|\Omega|\left((a_3)_-+(b_3)_-\right)\cr
 b_2>|\Omega|\left((a_4)_-+(b_4)_-\right).
\end{cases}
 \end{equation}
}

\noindent {\bf (H3)} {
\begin{equation}
\begin{cases}
\label{L1-boundedness-cond-eq-generalcase}
a_1-\frac{1}{2}\left((a_2)_-+(b_1)_-\right)-|\Omega|(a_3)_--\frac{1}{2}|\Omega|\left((a_4)_-+(b_3)_-\right)>0\cr
b_2-\frac{1}{2}\left((a_2)_-+(b_1)_-\right)-|\Omega|(b_4)_--\frac{1}{2}|\Omega|\left((a_4)_-+(b_3)_-\right)>0.
\end{cases}
 \end{equation}
}

We start by our main results on global existence of classical solutions.
\begin{theorem}
\label{thm-global-001}
\begin{itemize}
\item[(1)] Assume that {\bf (H1) } holds.
Then for any   $u_0,v_0 \in C^0(\bar{\Omega})$ with $u_0\ge 0$ and with $v_0\ge 0$ ,
$\eqref{u-v-w-eq1}$ has a unique  bounded  global  classical solution $(u(x,t;u_0,v_0),v(x,t;u_0,v_0)$, $w(x,t;u_0,v_0))$  which satisfies that
$$\lim_{t\to 0}\|u(\cdot,t;u_0,v_0)-u_0(\cdot)\|_{C^0(\bar\Omega)}=0,\quad \lim_{t\to 0}\|v(\cdot,t;u_0,v_0)-v_0(\cdot)\|_{C^0(\bar\Omega)}=0.
$$
 Moreover,
\[0\leq u(x,t) \leq \max\left\{\|u_0\|_\infty,M_{01} \right\}, \]
and
\[0\leq v(x,t)\leq  \max\left\{\|v_0\|_\infty,M_{02} \right\}, \]
with
\[M_{01}=\frac{a_0+\sqrt{a_0^2+4\big(a_1-k\frac{\chi_1}{d_3}-|\Omega|(a_3)_-\big)\big((a_2)_-+|\Omega|(a_4)_-+l\frac{\chi_1}{d_3}\big)M_{00}}}{2\big(a_1-k\frac{\chi_1}{d_3}-|\Omega|(a_3)_-\big)} ,\]
\[M_{02}=\frac{b_0+\sqrt{b_0^2+4\big(b_2-l\frac{\chi_2}{d_3}-|\Omega|(b_4)_-\big)\big((b_1)_-+|\Omega|(b_3)_-+k\frac{\chi_2}{d_3}\big)M_{00}}}{2\big(b_2-l\frac{\chi_2}{d_3}-|\Omega|(b_4)_-\big)} ,\]
\[M_{00}=\max\left\{\|u_0\|_\infty\|v_0\|_\infty\ ,\ \frac{(a_0+b_0)^2}{4L^2}  \right\},\]
where
\begin{align*}
L=\min\Big\{&a_1-k\frac{\chi_1+\chi_2}{d_3}-|\Omega|((a_3)_-+|(b_3)_-)-(b_1)_-,\\
&b_2-l\frac{\chi_2+\chi_1}{d_3}-|\Omega|((b_4)_-+(a_4)_-)-(a_2)_-\Big\}.
\end{align*}


\vspace{-0.1in}\item[(2)] Assume that {\bf (H2)} and
$$
\begin{cases}
a_1>\max\{0,\frac{\chi_1k(n-2)}{d_3n}\}\ ,\qquad a_2> \max\{0,\frac{\chi_1l(n-2)}{d_3n}\}, \\
b_2>\max\{0,\frac{\chi_2l(n-2)}{d_3n}\}\ ,\qquad b_1>\max\{0, \frac{\chi_2k(n-2)}{d_3n}\}
\end{cases} \qquad {\bf (H4)}
$$
hold.
 Then for any nonnegative initials $u_0,v_0 \in C^0(\bar{\Omega})$,  system  $\eqref{u-v-w-eq1}$ has a unique  bounded  global
  classical solution $(u(x,t;u_0,v_0)$, $v(x,t;u_0,v_0)$, $w(x,t;u_0,v_0))$ which satisfies that
$$\lim_{t\to 0}\|u(\cdot,t;u_0,v_0)-u_0(\cdot)\|_{C^0(\bar\Omega)}=0,\quad \lim_{t\to 0}\|v(\cdot,t;u_0,v_0)-v_0(\cdot)\|_{C^0(\bar\Omega)}=0.
$$
 Moreover,
\[0\leq\int_{\Omega}u(t) \leq \max\left\{\int_{\Omega}u_0,\frac{a_0+\sqrt{a_0^2+4\left(\frac{{a_1-|\Omega|(a_3)_-}}{|\Omega|}\right)(a_4)_-M}}{2\left(\frac{{a_1-|\Omega|(a_3)_-}}{|\Omega|}\right)} \right\} ,\]
and
\[0\leq\int_{\Omega}v(t)\leq \max\left\{\int_{\Omega}v_0,\frac{b_0+\sqrt{b_0^2+4\left(\frac{ b_2-|\Omega|(b_4)_-}{|\Omega|}\right)(b_3)_-M}}{2\left(\frac{ b_2-|\Omega|(b_4)_-}{|\Omega|}\right)} \right\}, \]
where
\[M=\max\left\{\|u_0\|_1 \|v_0\|_1,\,\,\frac{(a_0+b_0)^2|\Omega|^2}{4 \min\{(a_1-|\Omega|((a_3)_-+(b_3)_-))^2, (b_2-|\Omega|((b_4)_-+(a_4)_-))^2\}}\right\}.\]

\vspace{-0.1in}
\item[(3)] Assume that {\bf (H3)} holds. If in addition, either {\bf (H4)} holds or
$$
\min\left\{a_1-\left((a_{2})_-+\frac{(l+k)\chi_1}{d_3}\right)\ ,\ b_2-\left((b_{1})_- +\frac{(l+k)\chi_2}{d_3}\right) \right\}>0, \quad {\bf (H5)}
$$
or
$$\begin{cases}
 a_1&> \frac{n(a_2)_-}{n+2}+\frac{2(b_1)_-}{n+2}+\frac{\chi_1k(n-2)}{d_3n}+\frac{\chi_1l(n-2)}{d_3(n+2)}+2\frac{\chi_2k(n-2)}{d_3n(n+2)}
\nonumber\\
b_2&> \frac{n(b_1)_-}{n+2}+\frac{2(a_2)_-}{n+2}+\frac{\chi_2l(n-2)}{d_3n}+\frac{\chi_2l(n-2)}{d_3(n+2)}+2\frac{\chi_1l(n-2)}{d_3n(n+2)}
\end{cases},\quad  {\bf (H6)}
$$
holds, then for any  $u_0,v_0 \in C^0(\bar{\Omega})$  with $u_0\ge 0$ and $v_0\ge 0$,  system  $\eqref{u-v-w-eq1}$ has a unique bounded  global
  classical solution $(u(x,t;u_0,v_0)$, $v(x,t;u_0,v_0)$, $w(x,t;u_0,v_0))$ which satisfies that
$$\lim_{t\to 0}\|u(\cdot,t;u_0,v_0)-u_0(\cdot)\|_{C^0(\bar\Omega)}=0,\quad \lim_{t\to 0}\|v(\cdot,t;u_0,v_0)-v_0(\cdot)\|_{C^0(\bar\Omega)}=0.
$$

Furthermore
\begin{align*}
0\leq \int_{\Omega}u(t)+\int_{\Omega}v(t)  \leq \max \left\{  \int_{\Omega}u_0+\int_{\Omega}v_0, \,\, 2|\Omega|\frac{\max\{a_0,b_0\}}{\min\{\alpha,\beta\}}  \right\}
\quad \forall t >0,
\end{align*}
where
$$
\alpha=a_1-\frac{1}{2}\Big((a_2)_-+(b_1)_-+|\Omega|\left((a_4)_-+(b_3)_-\right)\Big)-|\Omega|(a_3)_-,
$$
and
$$
\beta= b_2-\frac{1}{2}\Big((a_2)_-+(b_1)_-+|\Omega|\left((a_4)_-+(b_3)_-\right)\Big)-|\Omega|(b_4)_-.
$$
\end{itemize}
\end{theorem}

\medskip

\begin{remark}
\begin{itemize}
\vspace{-0.1in}\item[(1)] When  system $\eqref{u-v-w-eq1}$ is considered in competitive case, that is,
 $a_i,b_i > 0$, then
{\bf (H2)} and {\bf (H3)} hold trivially. If in addition,  $n=1$ or $n=2,$ then hypotheses
{\bf (H5)} and {\bf (H6)} hold trivially. In this case, Theorem \ref{thm-global-001} (2) and Theorem \ref{thm-global-001} (3) imply that, for every nonnegative bounded and uniformly continuous initials  $(u_0,v_0)$, \eqref{u-v-w-eq1} has a unique bounded global classical solution.  This rules out the blow-up as for the case of one species.

\vspace{-0.1in}\item[(2)] When  system $\eqref{u-v-w-eq1}$ is considered in the competitive case, it follows from Theorem \ref{thm-global-001}(2) and Theorem \ref{thm-global-001}(3) that if {\bf (H4)} holds, then for  every nonnegative bounded and uniformly continuous initials  $(u_0,v_0)$, \eqref{u-v-w-eq1} has a unique bounded global classical solution. It remains open whether under hypothesis {\bf (H4)}, \eqref{u-v-w-eq1} has global bounded classical solution for every nonnegative initials $u_0,v_0\in C^0(\overline{\Omega})$ in the global cooperative case.


\vspace{-0.1in}\item[(3)] If $a_2\leq 0$ and $b_1\leq 0,$ or $a_2,b_1$ very small, then hypothesis {\bf (H4)} is never satisfied. In such case Theorem \ref{thm-global-001} (1) and Theorem \ref{thm-global-001} (3) provide sufficient condition on the conditions for the existence of bounded global classical solutions.

\vspace{-0.1in}\item[(4)] When  system $\eqref{u-v-w-eq1}$ is in the competitive case, {\bf (H1) } holds if and only if  $a_1>k\frac{\chi_1+\chi_2}{d_3}$ and $b_2>l\frac{\chi_1+\chi_2}{d_3}$. While {\bf (H5)} holds if and only if $a_1>\frac{(l+k)\chi_1}{d_3}$ and $b_2>\frac{(l+k)\chi_2}{d_3}$. In this case if either {\bf (H1)} or {\bf (H5)} holds, it follows from Theorem \ref{thm-global-001} that for every nonnegative bounded and uniformly continuous initials  $(u_0,v_0)$, \eqref{u-v-w-eq1} has a unique bounded global classical solution. Note the hypotheses {\bf (H1)} and {\bf (H5)} are not comparable.
\end{itemize}
\end{remark}



Next, we state our result on the phenomenon of coexistence in the general competitive-cooperative case.
\begin{theorem}
\label{Asym-Th-1}
 Assume that
 {\bf (H1)} holds, and suppose furthermore that
\begin{equation}
\label{Asymp-coexist-eq-01}
 a_1 >2\frac{\chi_1 }{d_3}k+|\Omega||a_3|,
\end{equation}
\begin{equation}
\label{Asymp-coexist-eq-02}
b_2 >2\frac{\chi_2 }{d_3}l+|\Omega||b_4|,
\end{equation}
\begin{equation}
\label{Asymp-coexist-eq-03}
\frac{a_2+|\Omega|a_4 }{b_2+|\Omega|b_4}< \frac{a_0}{b_0}<\frac{a_{1}+|\Omega|a_3}{b_1+|\Omega|b_3},
\end{equation}
 and
\begin{equation}
\label{Asymp-coexist-eq-4}
\begin{cases}
\left\{ a_1-2\frac{\chi_1}{d_3}k-|\Omega||a_3|\right\} \left\{ b_2-2\frac{\chi_2}{d_3}l-|\Omega||b_4| \right\} &\cr
>  \left(|a_2|+|\Omega||a_4|+l\frac{\chi_1}{d_3}\right)\left(|b_1|+|\Omega||b_3|+k\frac{\chi_2}{d_3}\right)&.
\end{cases}
\end{equation}
Then for every nonnegative initial functions $u_{0},v_0\in C^{0}(\overline{\Omega})$ satisfying $$\min\{\|u_0\|_{\infty};\|v_0\|_\infty\}>0, $$
\eqref{u-v-w-eq1} has a unique bounded and globally defined classical solution $$(u(\cdot,\cdot;u_0,v_0),v(\cdot,\cdot;u_0,v_0),w(\cdot,\cdot;u_0,v_0)).$$ Moreover, it holds that
 \begin{equation}\label{MainAsym-eq-01}
 \lim_{t\to\infty}\left\|u(\cdot,t;u_0,v_0)-u^*\right\|_\infty=0,
 \end{equation}
\begin{equation}\label{MainAsym-eq-02}
\lim_{t\to\infty}\left\|v(\cdot,t;u_0,v_0)-v^*\right\|_\infty=0,
\end{equation}
and
 \begin{equation}\label{MainAsym-eq-001}
 \lim_{t\to\infty}\left\|w(\cdot,t;u_0,v_0)-\frac{ku^*+lv^*}{\lambda}\right\|_\infty=0,
 \end{equation}

where
$$
u^*=\frac{a_0(b_2+|\Omega|b_4)-b_0(a_2+|\Omega|a_4)}{(b_2+|\Omega|b_4)(a_1+|\Omega|a_3)-(a_2+|\Omega|a_4)(b_1+|\Omega|b_3)},$$
and
$$ v^*=\frac{a_0(b_1+|\Omega|b_3)-b_0(a_1+|\Omega|a_3)}{(b_1+|\Omega|b_3)(a_2+|\Omega|a_4)-(a_1+|\Omega|a_3)(b_2+|\Omega|b_4)}.
$$
\end{theorem}

\medskip

\begin{remark}
\begin{itemize}
\vspace{-0.1in}\item[(1)] Note that hypothesis {\bf (H1)} assumed in Theorem \ref{Asym-Th-1} can be replaced by any hypothesis on the parameters under which global existence of bounded classical solution holds.

\vspace{-0.1in}\item[(2)] Condition \eqref{Asymp-coexist-eq-03} implies the semi trivial homogeneous solutions $(\frac{a_0}{a_1+|\Omega|a_3},0,\frac{ka_0}{\lambda(a_1+|\Omega|a_3)})$ and $(0,\frac{b_0}{b_2+|\Omega|b_4},\frac{lb_0}{\lambda(b_2+|\Omega|b_4)})$ are unstable.\\
 Indeed $b_0>\frac{a_0\left(b_1+|\Omega|b_3\right)}{a_1+|\Omega|a_3}$ implies  $(\frac{a_0}{a_1+|\Omega|a_3},0,\frac{ka_0}{\lambda(a_1+|\Omega|a_3)})$ is unstable and $a_0>\frac{b_0(a_2+|\Omega|a_4)}{b_2+|\Omega|b_4}$  implies  $(0,\frac{b_0}{b_2+|\Omega|b_4},\frac{lb_0}{\lambda(b_2+|\Omega|b_4)})$ is unstable.

\vspace{-0.1in}\item[(3)] In the case of system $\eqref{u-v-w-eq1bis}$, \eqref{Asymp-coexist-eq-03} becomes $\tilde a_2 <1$ and $\tilde b_1<1$, that is \eqref{Asymp-coexist-eq-03} indicates in general a weak competition. Furthermore conditions \eqref{Asymp-coexist-eq-01}, \eqref{Asymp-coexist-eq-02} and \eqref{Asymp-coexist-eq-4} becomes respectively $\mu_1> \frac{2\chi_1k}{d_3},$ $\mu_2> \frac{2\chi_2l}{d_3},$ and $\left(\mu_1-\frac{2\chi_1k}{d_3}\right)\left(\mu_2-\frac{2\chi_2l}{d_3}\right)>\left(\mu_1\tilde a_2+\frac{\chi_1l}{d_3}\right)\left(\mu_2\tilde b_1+\frac{\chi_2k}{d_3}\right).$   If in addition $\chi_1=\chi_2=0,$ all this last three conditions become $\tilde a_2\tilde b_2<1$ which is trivially true in this weak completion case of $\tilde a_2<1$ and $\tilde b_1<1.$
\end{itemize}
\end{remark}

\medskip

Following similar arguments as the proof of Theorem \ref{Asym-Th-1}, we can prove the following important  result on coexistence in the competitive case that $a_i, b_i>0$, $i=1,2,3,4$.
\begin{theorem}
\label{Exclusion-comp-Th-1}
Suppose \eqref{Asymp-coexist-eq-01}, \eqref{Asymp-coexist-eq-02}, \eqref{Asymp-coexist-eq-03},
\begin{equation}\label{Exclusion-comp-eq-00}
\min\{a_2-\frac{\chi_1}{d_3}l, b_1-\frac{\chi_2}{d_3}k\}\geq 0,
\end{equation}
 and
\begin{equation}
\label{Exclusion-comp-eq-01}
\left\{ a_1-2\frac{\chi_1}{d_3}k-|\Omega|a_3\right\} \left\{ b_2-2\frac{\chi_2}{d_3}l-|\Omega|b_4 \right\} >  (a_2+|\Omega|a_4)(b_1+|\Omega|b_3).
\end{equation}
Then for every nonnegative initial functions $u_{0},v_0\in C^{0}(\overline{\Omega})$ satisfying $$\min\{\|u_0\|_{\infty};\|v_0\|_\infty\}>0, $$
\eqref{u-v-w-eq1} has a unique bounded and globally defined classical solution $$(u(\cdot,\cdot;u_0,v_0),v(\cdot,\cdot;u_0,v_0),w(\cdot,\cdot;u_0,v_0))$$
 Moreover, it holds that
 \begin{equation*}
 \lim_{t\to\infty}\left\|u(\cdot,t;u_0;v_0)-u^*\right\|_\infty=0,
 \end{equation*}
\begin{equation*}
\lim_{t\to\infty}\left\|v(\cdot,\cdot;u_0;v_0)-v^*\right\|_\infty=0,
\end{equation*}
and
 \begin{equation*}
 \lim_{t\to\infty}\left\|w(\cdot,t;u_0,v_0)-\frac{ku^*+lv^*}{\lambda}\right\|_\infty=0.
 \end{equation*}

\end{theorem}

\medskip

\begin{remark}
 In the case of system $\eqref{u-v-w-eq1bis}$, conditions \eqref{Asymp-coexist-eq-01}, \eqref{Asymp-coexist-eq-02}, \eqref{Asymp-coexist-eq-03} and  \eqref{Exclusion-comp-eq-00} become condition \eqref{eq-int-002}. Furthermore condition \eqref{Exclusion-comp-eq-01} becomes  \eqref{eq-int-003}. Thus Theorem \ref{Exclusion-comp-Th-1} is consistent with the coexistence result in \cite{TBJLMM16}.
\end{remark}

Finally we state the main results on exclusion
\begin{theorem}
\label{Asym-Th-2}
  Assume that
{\bf (H1)}, and suppose furthermore that
\eqref{Asymp-coexist-eq-02} holds,  $a_4\geq 0,$
\begin{equation}
\label{Asymp-exclusion-eq-00}
 a_2\geq\frac{\chi_1 }{d_3}l,
\end{equation}
\begin{equation}
\label{Asymp-exclusion-eq-01}
  a_1 >\frac{\chi_1 }{d_3}k+|\Omega|(a_3)_-,
\end{equation}
\begin{equation}
\label{Asymp-exclusion-eq-02}
b_0 \geq \frac{b_2+|\Omega|b_4}{a_2+a_{4}|\Omega|}a_0,
\end{equation}
and
\begin{equation}\label{Asymp-exclusion-eq-03}
\begin{cases}
\left(a_{1}-\frac{\chi_{1}k}{d_{3}}-|\Omega|(a_3)_-\right)\left(b_{2}-2\frac{\chi_{2}}{d_{3}}l-|\Omega||b_{4}|\right) b_0&\cr
>\left (b_{2}-\frac{\chi_{2}}{d_{3}}l-|\Omega|(b_{4})_-\right)\left(b_{1}+b_{3}|\Omega|\right)a_0+|\Omega|(b_3)_-\left (b_{4}+\frac{\chi_{2}}{d_{3}}l\right)a_0 \,\text{if}  \quad b_1>\frac{\chi_2k}{d_3}&\cr

\left(a_{1}-\frac{\chi_{1}k}{d_{3}}-|\Omega|(a_3)_-\right)\left(b_{2}-2\frac{\chi_{2}}{d_{3}}l-|\Omega||b_{4}|\right) b_0&\cr
>\left (b_{2}-\frac{\chi_{2}}{d_{3}}l-|\Omega|(b_{4})_-\right)\left(|\Omega|(b_{3})_++\frac{\chi_{2}}{d_{3}}k\right)a_0&\cr
 +\left(\frac{\chi_{2}}{d_{3}}k-b_1+|\Omega|(b_3)_-\right)\left (b_{4}+\frac{\chi_{2}}{d_{3}}l\right)a_0\,\text{if} \, \,b_1\leq \frac{\chi_2k}{d_3} &
\end{cases}
 \end{equation}
Then for every nonnegative initial functions $u_{0},v_0\in C^{0}(\overline{\Omega}),$ $u_0\geq 0,$ $v_0\geq 0,$ with $\|v_0\|_\infty>0, $
\eqref{u-v-w-eq1} has a unique bounded and globally defined classical solution $$(u(\cdot,\cdot;u_0,v_0),v(\cdot,\cdot;u_0;v_0),w(\cdot,\cdot;u_0,v_0)).$$ Moreover, it holds that
 \begin{equation}\label{MainAsym-eq-001}
 \lim_{t\to\infty}\left\|u(\cdot,t;u_0;v_0)\right\|_\infty=0,
 \end{equation}
\begin{equation}\label{MainAsym-eq-02}
\lim_{t\to\infty}\left\|v(\cdot,\cdot;u_0;v_0)-\frac{b_0}{b_2+|\Omega|b_4}\right\|_\infty=0,
\end{equation}
and
 \begin{equation}\label{MainAsym-eq-002}
\lim_{t\to\infty}\left\|w(\cdot,\cdot;u_0;v_0)-\frac{lb_0}{\lambda(b_2+|\Omega|b_4)}\right\|_\infty=0.
\end{equation}

\end{theorem}

\begin{remark}
\begin{itemize}
\vspace{-0.1in}\item[(1)]  The condition {\bf (H1)} is needed in Theorem \ref{Asym-Th-2} only when $b_1\leq \frac{\chi_2k}{d_3}.$ Indeed the hypothesis {\bf (H1)} assumed in Theorem \ref{Asym-Th-2} can be replaced by any hypothesis on the parameters under which global existence of bounded classical solution holds.

\vspace{-0.1in}\item[(2)]In the case of system \eqref{u-v-w-eq1bis}, \eqref{Asymp-coexist-eq-02}, \eqref{Asymp-exclusion-eq-00},  \eqref{Asymp-exclusion-eq-01}, \eqref{Asymp-exclusion-eq-02} become respectively  $q_2<\frac{1}{2}$, $q_1\leq \tilde a_1$ and  $kq_1<1,$ and $1\leq \tilde a_2$ . Furthermore  \eqref{Asymp-exclusion-eq-03} become
\begin{equation*}\label{Asymp-exclusion-eq-03bis}
\begin{cases}
q_1k +(2-\tilde b_1)q_2+\tilde b_1 -2kq_1q_2<1,\,\text{if}  \quad \tilde b_1>kq_2\cr

q_1k +(2+k-\tilde b_1)q_2-2kq_1q_2<1,\,\text{if}  \quad \tilde b_1\leq kq_2.
\end{cases}
 \end{equation*}
Thus Theorem \ref{Asym-Th-2} is consistent with the exclusion result in \cite{STW13}.
\end{itemize}
\end{remark}

The rest of the paper is organized as follows. In section 1, we study the global existence of classical solutions and prove Theorem \ref{thm-global-001}. Section 2 is devoted to the study of the asymptotic behaviors of globally defined classical solutions. It is here that we present the proofs of Theorems \ref{Asym-Th-1} and \ref{Asym-Th-2}.

\section{Global Existence}
\label{S:Global}

In this section we study the global existence of  classical solutions to \eqref{u-v-w-eq1} and prove Theorem \ref{thm-global-001}. We start with the following important result on the local existence of classical solutions for any given  nonnegative bounded and uniformly continuous initials.

\begin{lemma}
\label{lm-local-001}
 For any given $u_0,v_0 \in C^0(\bar{\Omega})$ with $u_0 \geq 0$ and $v_0 \geq 0$,  there exists $T_{\max} \in (0,\infty]$  such that $\eqref{u-v-w-eq1}$  has a unique non-negative classical solution
$(u(x,t;u_0,v_0),v(x,t;u_0,v_0)$, $w(x,t;u_0,v_0))$ on $(0,T_{\max})$ satisfying that
$$\lim_{t\to 0^+}\|u(\cdot,t;u_0,v_0)-u_0(\cdot)\|_{C^0(\bar\Omega)}=0,\quad
\lim_{t\to
0^+}\|v(\cdot,t;u_0,v_0)-v_0(\cdot)\|_{C^0(\bar\Omega)}=0, $$
 and
moreover if $T_{\max}< \infty,$ then
\begin{equation}
\label{local-eq2}
\limsup_{t \nearrow T_{\max}}\left( \left\| u(\cdot,t;u_0,v_0) \right\|_{C^0(\bar \Omega)} +\left\| v(\cdot,t;u_0,v_0) \right\|_{C^0(\bar \Omega)}\right) =\infty.
\end{equation}
\end{lemma}

The proof of Lemma \ref{local-eq2} follows from standard arguments from fixed point theory or semigroup theory and regularity arguments (see for example \cite[proof of  Lemma 2.1]{STW13} and [proof of Theorem 1.1 ]\cite{ITBWS16}).

Our approach to prove our main result   on the existence of classical solutions which are globally defined in time Theorem \ref{thm-global-001} is as follows. For Theorem \ref{thm-global-001}(2) and Theorem \ref{thm-global-001}(3),  we use the celebrate method  of  $L^p$ estimates and Gagliardo-Nirengerg's Inequality. While for Theorem \ref{thm-global-001}(1), we use the rectangles method which relies on the dynamics of the following system of ODE's induced by system \eqref{u-v-w-eq1}.

{\large
\begin{equation}
\label{ode00}
\begin{cases}
\overline{u}'=\frac{\chi_1}{d_3} \overline{u}\big(k \overline {u}+l\overline v-k\underline{u}-l\underline{v}\big)+ \overline{u}\big[a_0-\big(a_1-|\Omega|(a_3)_-\big)\overline u-
|\Omega|(a_3)_+ \underline{u}\big]\\
\qquad +\overline{u}\big[\big((a_2)_-+|\Omega|(a_4)_-\big)\overline{ v}-\big((a_2)_++|\Omega|(a_4)_+\big)\underline{ v}\big]\\
\\
\underline{u}'=\frac{\chi_1}{d_3} \underline{u}\big(k\underline {u}+l\underline {v}-k\overline{u}-l\overline{v}\big)+ \underline{u}\big[a_0-(a_1-|\Omega|(a_3)_-)\underline u-|\Omega|(a_3)_+ \overline{u}\big]\\
\qquad +\underline{u}\big[\big((a_2)_-+|\Omega|(a_4)_-\big)\underline{ v}-\big((a_2)_++|\Omega|(a_4)_+\big)\overline{ v}\big]\\

\\
\overline{v}'=\frac{\chi_2}{d_3} \overline{v}\big(k\overline {u}+l\overline {v}-k\underline{u}-l\underline{v}\big)+ \overline{v}\big[b_0-\big(b_2-|\Omega|(b_4)_-\big)\overline{v}-
|\Omega|(b_4)_+ \underline{v}\big]\\
\qquad +\overline{v}\big[\big((b_1)_-+|\Omega|(b_3)_-\big)\overline{ u}-\big((b_1)_++|\Omega|(b_3)_+\big)\underline{ u}\big]\\
\\

\underline{v}'=\frac{\chi_2}{d_3} \underline{v}\big(k\underline {u}+l\underline {v}-k\overline{u}-l\overline{v}\big)+ \underline{v}\big[b_0-(b_2-|\Omega|(b_4)_-)\underline{v}-
|\Omega|(b_4)_+ \overline{v}\big]\\
\qquad +\underline{v}\big[\big((b_1)_-+|\Omega|(b_3)_-\big)\underline{ u}-\big((b_1)_++|\Omega|(b_3)_+\big)\overline{ u}\big].\\
\end{cases}
\end{equation}
}
Note that for that for every nonnegative real numbers $(\overline{u}_0,\underline{u}_{0},\overline{v}_0,\underline{v}_0)\in \R^{4}_{+}$, system \eqref{ode00} has a unique nonnegative classical solution $(\overline{u}(\cdot),\underline{u}(\cdot),\overline{v}(\cdot),\underline{v}(\cdot))$ with $ (\overline{u}(0),\underline{u}(0),\overline{v}(0),\underline{v}(0))=(\overline{u}_0,\underline{u}_{0},\overline{v}_0,\underline{v}_0)$ defined on a maximal interval $[0\ ,\ T_{max}(\overline{u}_0,\underline{u}_{0},\overline{v}_0,\underline{v}_0))$.

 Furthermore, if $ T_{max}(\overline{u}_0,\underline{u}_{0},\overline{v}_0,\underline{v}_0))<\infty$, then
\begin{equation}\label{crit-for-Tmax}
\lim_{t\to T_{max}(\underline{u}_{0},\overline{u}_0,\underline{v}_0,\overline{v}_0)_{-}}\Big(|\overline{u}(t)|+|\underline{u}(t)|+|\overline{v}(t)|+|\underline{v}(t)| \Big)=\infty.
\end{equation}
For given $\varphi_0\in C^0(\bar\Omega)$ with $\varphi_0(x)\ge 0,$  we let
$\overline{\varphi}_0=\max_{x \in \bar{\Omega}} \varphi_0(x)$ and
$\underline{\varphi}_0=\min_{x \in \bar{\Omega}} \varphi_0(x)$.



We start by the following two Lemmas which provide a sufficient condition for solutions of system \eqref{ode00} to be defined for all time.

\begin{lemma}
\label{lem-0000}
Let $A_1>(B_1)_{+}$ and $B_2>(A_2)_{+}$ be given real numbers. Let $(\overline{u}(t),\overline{v}(t))$ be positive continuously differentiable function on $[0, T_{\max})$, $T_{max}\in(0\ ,\ \infty]$,  and satisfying the system of differential inequalities
\begin{equation}
\label{ode01}
\begin{cases}
\overline{u}_t\leq \overline{u}\left(a_0-A_1\overline{u}+A_2\overline{v}\right)\cr
\overline{v}_t\leq \overline{v}\left(b_0+B_1\overline{u}-B_2\overline{v}\right).
\end{cases}
\end{equation}
Then the function $(\overline{u}(t),\overline{v}(t))$ satisfies
\[0\leq\overline{u}(t) \leq \max\left\{\overline{u}_0,\frac{a_0+\sqrt{a_0^2+4A_1(A_2)_{+}M}}{2A_1} \right\}, \quad \forall\ 0\leq t< T_{max},\]
and
\[0\leq\overline{v}(t) \leq \max\left\{\overline{v}_0,\frac{b_0+\sqrt{b_0^2+4(B_1)_{+}B_2M}}{2B_1}\right\}, \quad \forall\ 0\leq t< T_{max}, \]
where
\begin{equation}\label{M-eq}
M= \max\left\{\overline{u}_0\overline{v}_0, \frac{(a_0+b_0)^2}{4\min\{(A_1-B_1)^2,(B_2-A_2)^2\}} \right\}.
\end{equation}
\end{lemma}

\begin{proof}
We have from   \eqref{ode01} that
\[
\begin{cases}
\frac{\overline{u}_t}{\overline{u}}\leq a_0-A_1\overline{u}+A_2\overline{v}\cr
\frac{\overline{v}_t}{\overline{v}}\leq b_0+B_1\overline{u}-B_2\overline{v}.
\end{cases}
\]
Thus supposing $A_1>B_1$ and $B_2>A_2,$  we get
\begin{align*}
\frac{\overline{u}_t}{\overline{u}}+\frac{\overline{v}_t}{\overline{v}}&\leq a_0+b_0-(A_1-B_1)\overline{u}-(B_2-A_2)\overline{v}\nonumber\\
&\leq a_0+b_0- \min\{A_1-B_1,B_2-A_2\}(\overline{u}+\overline{v})\nonumber\\
&\leq  a_0+b_0-2 \min\{A_1-B_1,B_2-A_2\}(\sqrt{\overline{u}\, \overline{v}}).
\end{align*}
Thus, it follows from the last inequality that
\begin{align*}
\Big(\ln{(\overline{u}\, \overline{v})}\Big)'&\leq  a_0+b_0-2
\min\{A_1-B_1,B_2-A_2\}e^{\frac{1}{2}\ln{(\overline{u}\,\overline{v})}}.
\end{align*}
Therefore by comparison principle for ODEs, we get
\[\ln(\overline{u}\, \overline{v})\leq \max\{\ln{(\overline{u}_0\overline{v}_0)},2\ln\{\frac{a_0+b_0}{2 \min\{A_1-B_1,B_2-A_2\}}\}\}.\]
Hence  $0<\overline{u}\, \overline{v}\leq M$, where $M$ is given by \eqref{M-eq}. Combining this with \eqref{ode01} we get
\[
\begin{cases}
\overline{u}_t\leq a_0\overline{u}-A_1\overline{u}^2+(A_2)_{+}M\cr
\overline{v}_t\leq b_0\overline{v}-B_2\overline{v}^2+(B_1)_{+}M.
\end{cases}
\]
Therefore, again by comparison principle for ODEs, we get
\[0<\overline{u}(t) \leq \max\{\overline{u}_0,\frac{a_0+\sqrt{a_0^2+4A_1(A_2)_{+}M}}{2A_1} \}, \]
and
\[0<\overline{v}(t) \leq \max\{\overline{v}_0,\frac{b_0+\sqrt{b_0^2+4(B_1)_{+}B_2M}}{2B_2} \}. \]
The lemma thus follows.
\end{proof}
\begin{lemma}
\label{lem-0001}
Let $(\overline{u}(t),\underline{u}(t),\overline{v}(t),\underline{v}(t))$ be the solution of \eqref{ode00} with initial condition  $(\overline{u}(0),\underline{u}(0)$, $\overline{v}(0),\underline{v}(0))=(\overline{u}_0,\underline{u}_0,\overline{v}_0,\underline{v}_0)\in\R^4_+$
in $(0, T_{\max}(\overline{u}_0,\underline{u}_0,\overline{v}_0,\underline{v}_0))$. If $\overline{u}_0\geq \underline{u}_0$ and $\overline{v}_0\geq\underline{v}_0$, then we have that
$$
0\leq \underline{u}(t) \leq \overline{u}(t) \quad \text{and}\quad 0\leq \underline{v}(t) \leq \overline{v}(t) \quad  \forall t : 0\leq t < T_{\max}(\overline{u}_0,\underline{u}_0,\overline{v}_0,\underline{v}_0).
$$
If in addition, {\bf (H1)} holds,  then   $T_{max}(\overline{u}_0,\underline{u}_0,\overline{v}_0,\underline{v}_0)=\infty$ and we have
\[0\leq \underline{u}(t) \leq \overline{u}(t) \leq \max\left\{\overline{u}_0,\frac{a_0+\sqrt{a_0^2+4\big(a_1-k\frac{\chi_1}{d_3}-|\Omega|(a_3)_-\big)\big((a_2)_-+|\Omega|(a_4)_-+l\frac{\chi_1}{d_3}\big)M}}{2\big(a_1-k\frac{\chi_1}{d_3}-|\Omega|(a_3)_-\big)} \right\},\]
and
\[0\leq \underline{v}(t) \leq \overline{v}(t)\leq  \max\left\{\overline{v}_0,\frac{b_0+\sqrt{b_0^2+4\big(b_2-l\frac{\chi_2}{d_3}-|\Omega|(b_4)_-\big)\big((b_1)_-+|\Omega|(b_3)_-+k\frac{\chi_2}{d_3}\big)M}}{2\big(b_2-l\frac{\chi_2}{d_3}-|\Omega|(b_4)_-\big)} \right\}, \]
with
\begin{equation}\label{M1}
M=\max\left\{\overline{u}_0\overline{v}_0,\,\, \frac{(a_0+b_0)^2}{4 L^2}\right\} ,
\end{equation}
where
\begin{align*}
L=\min\Big\{&a_1-k\frac{\chi_1+\chi_2}{d_3}-|\Omega|((a_3)_-+|(b_3)_-)-(b_1)_-,\\
&b_2-l\frac{\chi_2+\chi_1}{d_3}-|\Omega|((b_4)_-+(a_4)_-)-(a_2)_-\Big\}.
\end{align*}
\end{lemma}

\begin{proof}
Without Loss of Generality we may suppose that $0<
\underline{u}_0<\overline{u}_0 $ and $0<
\underline{v}_0<\overline{v}_0,$ since the result in the general case
follows from continuity of solutions with respect to initial
conditions. Suppose by contradiction that the result of Lemma
\ref{lem-0001} does not hold. Then there exists $\bar{t} \in (0\ ,\
T_{\max}(\overline{u}_0,\underline{u}_0,\overline{v}_0,\underline{v}_0))$ such that
$$
0< \underline{u}(t) < \overline{u}(t) \quad \text{and} \quad 0< \underline{v}(t) < \overline{v}(t) \quad \text{for all } \, 0<t<\bar{t},
$$
and either

{ \bf Case I.} $
\underline{u}(\bar{t}) =\overline{u}(\bar{t}) \quad \text{and} \quad \underline{v}(\bar{t}) < \overline{v}(\bar{t});
$

or

{ \bf Case II.}
$
\quad \underline{u}(\bar{t}) < \overline{u}(\bar{t}) \quad \text{and} \quad  \underline{v}(\bar{t}) =\overline{v}(\bar{t});
$

or

{ \bf Case III.}
$
\underline{u}(\bar{t}) =\overline{u}(\bar{t}) \quad \text{and} \quad \underline{v}(\bar{t}) = \overline{v}(\bar{t}).
$

\medskip

\noindent{\bf Case III.} It cannot happen, for otherwise, by uniqueness of solutions to system \eqref{ode00} we would have $(\overline{u}(t),\underline{u}(t),\overline{v}(t),\underline{v}(t))=(\underline{u}(t),\overline{u}(t),\underline{v}(t),\overline{v}(t))$ for all $t\in(0 ,T_{\max}(\overline{u}_0,\underline{u}_0,\overline{v}_0,\underline{v}_0))$, which contradicts the fact that $
0\leq \underline{u}(t) < \overline{u}(t) \quad \text{and} \quad 0\leq \underline{v}(t) < \overline{v}(t) \quad \text{for all } \, 0<t<\bar{t}.$

\medskip

\noindent{\bf Case I.} Suppose that $
\underline{u}(\bar{t}) =\overline{u}(\bar{t}) \quad \text{and} \quad \underline{v}(\bar{t}) < \overline{v}(\bar{t}).
$
Then $(\overline{u}-\underline{u})'(\bar{t})\leq 0$ and from the first two equation of system \eqref{ode00} at $\bar{t}$ , we get
$$
(\overline{u}-\underline{u})'(\bar{t})=\overline{u}(\bar{t})\left[2\frac{l\chi_1}{d_3}+|a_2|+|\Omega||a_4|\right](\overline{v}(\bar{t})-\underline{v}(\bar{t})).
$$
Since $0<\overline{u}(\bar{t})$ and $0<(\overline{v}(\bar{t})-\underline{v}(\bar{t})),$ we get  $(\overline{u}-\underline{u})'(\bar{t})>0$, which is a contradiction.

\medskip
\noindent{ \bf Case II.}  A similar argument as in  { \bf Case I.}  implies that {\bf Case II.} cannot happen.

From the first and third equations of system \eqref{ode00}  we get

{\large

\begin{equation*}
\begin{cases}
\overline{u}'\leq \overline{u}\big[a_0-\underbrace{\big(a_1-k\frac{\chi_1}{d_3}-|\Omega|(a_3)_-\big)}_{A_1}\overline u+\underbrace{\big((a_2)_-+|\Omega|(a_4)_-+l\frac{\chi_1}{d_3}\big)}_{A_2}\overline{ v}\big]\\

\overline{v}'\leq \overline{v}\big[b_0+\underbrace{\big((b_1)_-+|\Omega|(b_3)_-+k\frac{\chi_2}{d_3}\big)}_{B_1}\overline{ u}-\underbrace{\big(b_2-l\frac{\chi_2}{d_3}-|\Omega|(b_4)_-\big)}_{B_2}\overline{v}\big].\\

\end{cases}
\end{equation*}
}
Then  Lemma \ref{lem-0000} and condition {\bf (H1)} give
\[0\leq \underline{u}(t) \leq \overline{u}(t) \leq \max\left\{\overline{u}_0,\frac{a_0+\sqrt{a_0^2+4\big(a_1-k\frac{\chi_1}{d_3}-|\Omega|(a_3)_-\big)\big((a_2)_-+|\Omega|(a_4)_-+l\frac{\chi_1}{d_3}\big)M}}{2\big(a_1-k\frac{\chi_1}{d_3}-|\Omega|(a_3)_-\big)} \right\}, \]
and
\[0\leq \underline{v}(t) \leq \overline{v}(t)\leq  \max\left\{\overline{v}_0,\frac{b_0+\sqrt{b_0^2+4\big(b_2-l\frac{\chi_2}{d_3}-|\Omega|(b_4)_-\big)\big((b_1)_-+|\Omega|(b_3)_-+k\frac{\chi_2}{d_3}\big)M}}{2\big(b_2-l\frac{\chi_2}{d_3}-|\Omega|(b_4)_-\big)} \right\}, \]
where $M$ is given by \eqref{M1}. Thus, we must have that $T_{max}(\overline{u}_0,\underline{u}_0,\overline{v}_0,\underline{v}_0)=\infty$ and Lemma \ref{lem-0001} thus follows.
\end{proof}

The next two Lemmas give  a uniform $L^{1}$-bound for the solutions of \eqref{u-v-w-eq1} under hypotheses {\bf (H2)} and {\bf (H3)}, respectively.

\begin{lemma}
\label{lem-0002}
Suppose that we have the local competitive case, that is,  $a_2\geq 0$ and $b_1\geq 0$ (thus $(a_2)_-=(b_1)_-=0$) and suppose {\bf (H2)} holds.

Let $(u,v,w)$ be the solution of system \eqref{u-v-w-eq1} with initial
condition  $(u(0),v(0))=(u_0,v_0)$ in $(0,T_{\max}).$ Then, for every $0<t<T_{max}$, there holds
\[0\leq\int_{\Omega}u(t) \leq M_0:=\max\left\{\int_{\Omega}u_0,\frac{a_0+\sqrt{a_0^2+4\left(\frac{{a_1-|\Omega|(a_3)_-}}{|\Omega|}\right)(a_4)_-M}}{2\left(\frac{{a_1-|\Omega|(a_3)_-}}{|\Omega|}\right)} \right\}, \]
and
\[0\leq\int_{\Omega}v(t)\leq M_1:=\max\left\{\int_{\Omega}v_0,\,\, \frac{b_0+\sqrt{b_0^2+4\left(\frac{ b_2-|\Omega|(b_4)_-}{|\Omega|}\right)(b_3)_-M}}{2\left(\frac{ b_2-|\Omega|(b_3)_-}{|\Omega|}\right)} \right\}, \]
where
\begin{equation}\label{M2}
M=\max\left\{ \|u_0\|_{1}\|v_0\|_{1}\ ,\ \frac{(a_0+b_0)^2|\Omega|^2}{4\min\{ (a_1-|\Omega|((a_3)_- +(b_3)_-))^2 , (b_2-|\Omega|((b_4)_- +(a_4)_-))^2\}}\right\}.
\end{equation}

\end{lemma}

\begin{proof}
By integrating with respect to $x$ the first two equations of system  \eqref{u-v-w-eq1} we get
\begin{equation}
\label{L1-bound}
\begin{cases}
\frac{d}{dt}\int_{\Omega}u=\int_{\Omega}u\left(a_0-a_1u-a_2v-a_3\int_{\Omega}u-a_4\int_{\Omega}v\right),\quad x\in \Omega \quad\cr
\frac{d}{dt}\int_{\Omega}v=\int_{\Omega}v\left(b_0-b_1u-b_2v-b_3\int_{\Omega}u-b_4\int_{\Omega}v\right),\quad x\in \Omega. \quad\cr
\end{cases}
\end{equation}
Since $a_{2}\geq 0$ and  $b_1\geq 0$, by H\"older's inequality, it follows from \eqref{L1-bound} that
\[
\begin{cases}
\frac{d}{dt}\int_{\Omega}u\leq\int_{\Omega}u\left(a_0-\frac{\left(a_1-|\Omega|(a_3)_-\right)}{|\Omega|}\int_{\Omega}u+(a_4)_-\int_{\Omega}v\right)\\
\frac{d}{dt}\int_{\Omega}v\leq\int_{\Omega}v\left(b_0-\frac{\left(b_2-|\Omega|(b_4)_-\right)}{|\Omega|}\int_{\Omega}v+(b_3)_-\int_{\Omega}u\right).\\
\end{cases}
\]
Then   Lemma \ref{lem-0000} and condition {\bf (H2)} give
\[0\leq\int_{\Omega}u(t) \leq \max\left\{\int_{\Omega}u_0,\frac{a_0+\sqrt{a_0^2+4\left(\frac{{a_1-|\Omega|(a_3)_-}}{|\Omega|}\right)(a_4)_-M}}{2\left(\frac{{a_1-|\Omega|(a_3)_-}}{|\Omega|}\right)} \right\} ,\]
and
\[0\leq\int_{\Omega}v(t)\leq \max\left\{\int_{\Omega}v_0,\frac{b_0+\sqrt{b_0^2+4\left(\frac{ b_2-|\Omega|(b_4)_-}{|\Omega|}\right)(b_3)_-M}}{2\left(\frac{ b_2-|\Omega|(b_4)_-}{|\Omega|}\right)} \right\} ,\]
where $M$ is given by \eqref{M2}.

\end{proof}

\begin{lemma}
\label{lem-0003}
Suppose {\bf (H3)} holds. Let $(u,v,w)$ be the solution of \eqref{u-v-w-eq1} with initial condition  $(u(0),v(0))=(u_0,v_0)$
in $(0,T_{\max})$. Then
\begin{align*}
0\leq \int_{\Omega}u(t)+\int_{\Omega}v(t) \leq \max \left\{  \int_{\Omega}u_0+\int_{\Omega}v_0,  \,\,2|\Omega|\frac{\max\{a_0,b_0\}}{\min\{\alpha,\beta\}} \right\},\quad \forall t \in (0, T_{max}),
\end{align*}
where
\begin{equation}\label{alpha-def}
\alpha=a_1-\frac{1}{2}\Big((a_2)_-+(b_1)_-+|\Omega|\left((a_4)_-+(b_3)_-\right)\Big)-|\Omega|(a_3)_-,
\end{equation}
and
\begin{equation}\label{beta-def}
\beta= b_2-\frac{1}{2}\Big((a_2)_-+(b_1)_-+|\Omega|\left((a_4)_-+(b_3)_-\right)\Big)-|\Omega|(b_4)_-.
\end{equation}
\end{lemma}

\begin{proof}
From \eqref{L1-bound}, we get
\[
\begin{cases}
\frac{d}{dt}\int_{\Omega}u\leq\int_{\Omega}u\left(a_0-a_1u+(a_2)_-v+(a_3)_-\int_{\Omega}u+(a_4)_-\int_{\Omega}v\right)\\
\frac{d}{dt}\int_{\Omega}v\leq\int_{\Omega}v\left(b_0-b_2v+(b_1)_-u+(b_3)_-\int_{\Omega}u+(b_4)_-\int_{\Omega}v\right).\\
\end{cases}
\]
By adding the two above equations, we get
\begin{align*}
\frac{d}{dt}\int_{\Omega}(u+v) &\leq \max\{a_0,b_0\}\int_{\Omega}(u+v)-a_1\int_{\Omega}u^2+(a_3)_-\left(\int_{\Omega}u\right)^2-b_2\int_{\Omega}v^2+(b_4)_-\left(\int_{\Omega}v\right)^2\nonumber\\
&+\left((a_2)_-+(b_1)_-\right)\int_{\Omega}uv+\left((a_4)_-+(b_3)_-\right)\left(\int_{\Omega}u\right)\left(\int_{\Omega}v\right).
\end{align*}
Then by young's Inequality, we get
\begin{align*}
\frac{d}{dt}\left(\int_{\Omega}u+\int_{\Omega}v\right)& \leq \max\{a_0,b_0\}\left(\int_{\Omega}u+\int_{\Omega}v\right)-a_1\int_{\Omega}u^2+(a_3)_-\left(\int_{\Omega}u\right)^2\\
&-b_2\int_{\Omega}v^2+(b_4)_-\left(\int_{\Omega}v\right)^2+\frac{1}{2}\left((a_2)_-+(b_1)_-\right)\left(\int_{\Omega}u^2+\int_{\Omega}v^2\right).\\
&+\frac{1}{2}\left((a_4)_-+(b_3)_-\right)\left(\left(\int_{\Omega}u\right)^2+\left(\int_{\Omega}v\right)^2\right).
\end{align*}
Hence, from H\"older's inequality and the last  inequality, it follows that
\begin{equation}\label{L1-bound01}
\frac{d}{dt}\left(\int_{\Omega}u+\int_{\Omega}v\right) \leq \max\{a_0,b_0\}\left(\int_{\Omega}u+\int_{\Omega}v\right)-\frac{\alpha}{|\Omega|}\left(\int_{\Omega}u\right)^2-\frac{\beta}{|\Omega|}\left(\int_{\Omega}v\right)^2.
\end{equation}
Since
$\left(\int_{\Omega}u+\int_{\Omega}v\right)^2\leq 2\left(\left(\int_{\Omega}u\right)^2+\left(\int_{\Omega}v\right)^2\right)$,
it follows from inequality \eqref{L1-bound01} that
\begin{equation*}
\frac{d}{dt}\left(\int_{\Omega}u+\int_{\Omega}v\right) \leq \max\{a_0,b_0\}\left(\int_{\Omega}u+\int_{\Omega}v\right)-\frac{1}{2|\Omega|}\min\{\alpha,\beta\}\left(\int_{\Omega}u+\int_{\Omega}v\right)^2.
\end{equation*}
Observe that {\bf (H3)} is equivalent to $\min\{\alpha , \beta\}>0$. Then, we get by ODE's comparison that
$$
0\leq \int_{\Omega}u(t)+\int_{\Omega}v(t)\leq \max \left\{  \int_{\Omega}u_0+\int_{\Omega}v_0,  2|\Omega|\frac{\max\{a_0,b_0\}}{\min\{\alpha,\beta\}}  \right\}
$$
for all  $t \in (0, T_{\max})$.
\end{proof}

Next, we define the following two functions that we will see their importance in the upcoming lemmas. Set
\begin{equation}\label{f-function}
f(\gamma):=a_1-\frac{\gamma(a_2)_- }{\gamma+1}-\frac{(b_1)_-}{\gamma+1}-\frac{\chi_1k(\gamma-1)}{d_3\gamma}-\frac{\chi_1l(\gamma-1)}{d_3(\gamma+1)}-\frac{\chi_2k(\gamma-1)}{d_3\gamma(\gamma+1)},
\end{equation}
and
\begin{equation}\label{g-function}
g(\gamma):=b_2-\frac{\gamma(b_1)_- }{\gamma+1}-\frac{(a_2)_-}{\gamma+1}-\frac{\chi_2l(\gamma-1)}{d_3\gamma}-\frac{\chi_2k(\gamma-1)}{d_3(\gamma+1)}-\frac{\chi_1l(\gamma-1)}{d_3\gamma(\gamma+1)}.
\end{equation}
Note that the functions $f(\gamma)$ and $g(\gamma)$ are continuous  at every  $\gamma\neq -1$, $f(1)=a_1-\frac{1}{2}((a_2)_-+(b_1)_-)$ and $g(1)=b_2-\frac{1}{2}((b_1)_-+(a_2)_-)$. Hence if {\bf (H3)} holds, then $\min\{f(1),g(1)\}>0$. Furthermore we have $$
f(\frac{n}{2})>0 \iff  a_1> \frac{n(a_2)_-}{n+2}+\frac{2(b_1)_-}{n+2}+\frac{\chi_1k(n-2)}{d_3n}+\frac{\chi_1l(n-2)}{d_3(n+2)}+2\frac{\chi_2k(n-2)}{d_3n(n+2)},
$$
and
$$
g(\frac{n}{2})>0 \iff  b_2> \frac{n(b_1)_-}{n+2}+\frac{2(a_2)_-}{n+2}+\frac{\chi_2l(n-2)}{d_3n}+\frac{\chi_2l(n-2)}{d_3(n+2)}+2\frac{\chi_1l(n-2)}{d_3n(n+2)}.
$$

Now, we state and prove the following important lemma toward global existence of bounded solutions.
\begin{lemma}
\label{lem-0004}
 Let $(u,v,w)$ be the solution of \eqref{u-v-w-eq1} with initial condition  $(u(0),v(0))=(u_0,v_0)$
in $(0,T_{\max})$ and suppose that
\begin{equation}\label{L1-cond}
\sup_{0\leq t<T_{max}}\left(\int_{\Omega}u(t)+\int_{\Omega}v(t) \right)<\infty.
\end{equation}
For every $\gamma\neq -1$, let $f(\gamma)$ and $g(\gamma)$ be given by \eqref{f-function} and \eqref{g-function}, respectively.
\begin{enumerate}
\item[(1)]For every $\bar{\gamma}\geq 1$ such that $\min\{f(\bar{\gamma}),g(\bar{\gamma})\}>0$, there is $\gamma_0>\bar{\gamma}$ such that
\begin{align}\label{restate-eq1}
\sup_{1\leq \gamma\leq \gamma_0\ , 0\leq t<T_{max}}\left(\int_{\Omega}u^{\gamma}(t)+\int_{\Omega}v^{\gamma}(t)  \right)<\infty.
\end{align}
\item[(2)] If $min\{a_1,b_1,a_2,b_2\}>0$, then for every $\gamma_0\in[1\ ,\ \gamma^*)$, \eqref{restate-eq1} holds, where
$$\gamma^*:=\min\left\{\frac{\chi_1k}{(\chi_1k-a_1)_+}\ ,\ \frac{\chi_2l}{(\chi_2l-b_2)_+}\ ,\ \frac{\chi_1l}{(\chi_1l-a_2)_+}\ ,\ \frac{\chi_2k}{(\chi_2k-b_1)_+}\right\}.$$
\item[(3)] If $\min\{f(1),f(\frac{n}{2}),g(1),g(\frac{n}{2})\}>0$, then for every $\gamma_0>0$,  \eqref{restate-eq1} holds.
\item[(4)] If $min\{a_1,b_1,a_2,b_2\}>0$ and
$$\frac{n}{2}<\gamma^*:=\min\left\{\frac{\chi_1k}{(\chi_1k-a_1)_+}\ ,\ \frac{\chi_2l}{(\chi_2l-b_2)_+}\ ,\ \frac{\chi_1l}{(\chi_1l-a_2)_+}\ ,\ \frac{\chi_2k}{(\chi_2k-b_1)_+}\right\},$$ then for every $\gamma_0>0$,  \eqref{restate-eq1} holds.
\end{enumerate}

\end{lemma}

\begin{proof} First, from  \eqref{L1-cond},  we have that
\begin{equation}\label{zz001}
C:=\sup_{0\leq t< T_{max}}\left[a_0+b_0+(|a_{3}|+|b_3|)\|u(\cdot,t)\|_{L^{1}(\overline{\Omega})}+(|a_4|+|b_4|)\|v(\cdot,t)\|_{L^{1}(\overline{\Omega})} \right]<\infty.
\end{equation}
(1) Since for every $w\in C^{0}(\overline{\Omega})$, the function $\gamma\mapsto |\Omega|^{-\frac{1}{\gamma}}\|w\|_\gamma$ is nondecreasing, it is enough to show that there is $\gamma_0>\bar{\gamma}$ such that $$\sup_{0\leq t<T_{\max}}\left(\int_{\Omega}u^{\gamma_{0}}(t)+\int_{\Omega}v^{\gamma_{0}}(t)  \right) <\infty. $$

Let $\gamma \geq 1$, by multiplying the  first equation of  $\eqref{u-v-w-eq1}$  by $u^{\gamma-1}(t)$ and integrating with respect to $x,$ we have for $t\in (0,T_{\max})$ that
      \begin{align}\label{zz002}\frac{1}{\gamma}\frac{d}{dt}\int_{\Omega}u^{\gamma}(t)+\frac{4d_1(\gamma-1)}{\gamma^2}\int_{\Omega}  |\nabla u^{\frac{\gamma}{2}}(t)|^2& =(\gamma-1)\chi_1\int_{\Omega}u^{\gamma-1}(t)\nabla u (t)\cdot \nabla w(t)\nonumber \\
  &+\int_{\Omega}u^{\gamma}(t)\Big[a_0-a_1u(t)-a_2v-a_3\int_{\Omega}u(t)-a_4\int_{\Omega}v(t)  \Big].
       \end{align}
By multiplying the  third equation of  $\eqref{u-v-w-eq1}$  by
$u^{\gamma}(\cdot)$ and integrating over $\Omega,$ we get
\begin{align}\label{zz003}
\int_{\Omega}u^{\gamma-1}(t)\nabla u (t)\cdot
\nabla w(t) =-\frac{\lambda}{d_3\gamma}\int_{\Omega}w(t)
u^{\gamma}(t)
+\frac{k}{d_3\gamma}\int_{\Omega}u^{\gamma+1}(t) +
\frac{l}{d_3\gamma}\int_{\Omega}u^{\gamma}(t)v(t).
\end{align}
Thus, combining \eqref{zz001}, \eqref{zz002} and  \eqref{zz003}, we have for $t\in (0\ ,\ T_{\max})$ that
\begin{eqnarray*}
&&\frac{1}{\gamma}\frac{d}{dt}\int_{\Omega}u^{\gamma}(t)+\frac{4(\gamma-1)d_1}{\gamma^2}\int_{\Omega}  |\nabla u^{\frac{\gamma}{2}}(t)|^2\\
&& \leq -\left(a_1- \frac{\chi_1k(\gamma-1)}{d_3\gamma}\right)\int_{\Omega}u^{\gamma+1}(t) +\left(\left(a_2\right)_-+\frac{\chi_1l(\gamma-1)}{d_3\gamma}\right)\int_{\Omega}u^{\gamma}(t)v(t) +C\int_{\Omega}u^{\gamma}.
\end{eqnarray*}
Combining this with the fact that $\int_{\Omega}u^{\gamma}v\leq \frac{\gamma}{\gamma+1}\int_{\Omega}u^{\gamma+1}+ \frac{1}{\gamma+1}\int_{\Omega}v^{\gamma+1}$,   we obtain that
\begin{eqnarray*}
&&\frac{1}{\gamma}\frac{d}{dt}\int_{\Omega}u^{\gamma}(t)+\frac{4(\gamma-1)d_1}{\gamma^2}\int_{\Omega}  |\nabla u^{\frac{\gamma}{2}}(t)|^2\\
&& \leq -\left(a_1-\frac{\gamma(a_2)_- }{\gamma+1}-\frac{\chi_1k(\gamma-1)}{d_3\gamma}-\frac{\chi_1l(\gamma-1)}{d_3(\gamma+1)}\right)\int_{\Omega}u^{\gamma+1}(t)  \\
&& +\left(\frac{(a_2)_-}{\gamma+1}+\frac{\chi_1l(\gamma-1)}{d_3\gamma(\gamma+1)}\right)\int_{\Omega}v^{\gamma+1}+C\int_{\Omega}u^{\gamma}.
\end{eqnarray*}
Similarly, from the second and third inequalities of system \eqref{u-v-w-eq1} we get
\begin{eqnarray*}
&&\frac{1}{\gamma}\frac{d}{dt}\int_{\Omega}v^{\gamma}(t)+\frac{4(\gamma-1)d_2}{\gamma^2}\int_{\Omega}  |\nabla v^{\frac{\gamma}{2}}(t)|^2\\
&& \leq -\left(b_2-\frac{\gamma(b_1)_- }{\gamma+1}-\frac{\chi_2l(\gamma-1)}{d_3\gamma}-\frac{\chi_2k(\gamma-1)}{d_3(\gamma+1)}\right)\int_{\Omega}v^{\gamma+1}(t)  \\
&&+\left(\frac{(b_1)_-}{\gamma+1}+\frac{\chi_2k(\gamma-1)}{d_3\gamma(\gamma+1)}\right)\int_{\Omega}u^{\gamma+1}(t)+C\int_{\Omega}v^{\gamma}.
\end{eqnarray*}
By adding the two last equations, we get
\begin{eqnarray}\label{L0001}
&&\frac{1}{\gamma}\frac{d}{dt}\left(\int_{\Omega}u^{\gamma}+\int_{\Omega}v^{\gamma}\right)+\frac{4(\gamma-1)}{\gamma^2}\left(d_1\int_{\Omega}  |\nabla u^{\frac{\gamma}{2}}|^2+d_2\int_{\Omega}  |\nabla u^{\frac{\gamma}{2}}|^2\right)\nonumber\\
&& \leq -f(\gamma)\int_{\Omega}u^{\gamma+1}(t)  -g(\gamma)\int_{\Omega}v^{\gamma+1}(t) +C\left(\int_{\Omega}u^{\gamma}+\int_{\Omega}v^{\gamma}\right),
\end{eqnarray}
where $f(\gamma)$ and $g(\gamma)$ are given by \eqref{f-function} and \eqref{g-function} respectively. Since by our hypothesis, $f(\bar{\gamma})>0$, $g(\bar{\gamma})>0$, from the continuity of $f$ and $g$ at $\bar{\gamma}$, there exists $\gamma_0>\bar{\gamma}$ such that for any $\gamma \in [\bar{\gamma}, \gamma_0]$, we have  $f(\gamma)>0$ and  $g(\gamma)>0$. Let $\alpha_{\bar{\gamma}}=\min_{\gamma \in [\bar{\gamma},\gamma_0]}f(\gamma),$ and  $\beta_{\bar{\gamma}}=\min_{\gamma \in [\bar{\gamma},\gamma_0]}g(\gamma)$.
Note also that we have
\begin{align}\label{L0002}
\left(\int_{\Omega}u^{\gamma}+\int_{\Omega}v^{\gamma}\right)^\frac{\gamma+1}{\gamma}
&\leq 2^\frac{\gamma+1}{\gamma}|\Omega|^\frac{1}{\gamma}\left( \int_{\Omega}u^{\gamma+1}+ \int_{\Omega}v^{\gamma+1}\right) .
\end{align}

Thus, it follows from inequalities \eqref{L0001} and \eqref{L0002} that

\begin{eqnarray*}
&&\frac{1}{\gamma}\frac{d}{dt}\left(\int_{\Omega}u^{\gamma}+\int_{\Omega}v^{\gamma}\right)+\frac{4(\gamma-1)}{\gamma^2}\left(\int_{\Omega}  |\nabla u^{\frac{\gamma}{2}}|^2+\int_{\Omega}  |\nabla u^{\frac{\gamma}{2}}|^2\right)\\
&& \leq \left(C-\frac{1}{ 2^\frac{\gamma+1}{\gamma}|\Omega|^\frac{1}{\gamma}}\min\{\alpha_{\bar{\gamma}},\beta_{\bar{\gamma}} \}\left(\int_{\Omega}u^{\gamma}+\int_{\Omega}v^{\gamma}\right)^\frac{1}{\gamma}\right)\left(\int_{\Omega}u^{\gamma}+\int_{\Omega}v^{\gamma}\right).
\end{eqnarray*}
Therefore by comparison of ODEs, we get
$$ 0 \leq \int_{\Omega}u^{\gamma}+\int_{\Omega}v^{\gamma} \leq \max\left\{\int_{\Omega}u_0^{\gamma}+\int_{\Omega}v_0^{\gamma}\ ,\  \left(\frac{C2^{\frac{1+\gamma}{\gamma}}|\Omega|^{\frac{1}{\gamma}}}{\min\{\alpha_{\bar{\gamma}},\beta_{\bar{\gamma}}\}}\right)^{\gamma} \right\}, \quad \forall t \in (0,T_{max}).$$

\noindent(2) Suppose that $a_1>0$, $b_1>0$,   $a_2>0$, and  $b_2> 0$.  Next, take $\gamma^*$ to be
$$\gamma^*:=\min\left\{\frac{\chi_1k}{(\chi_1k-a_1)_+}\ ,\ \frac{\chi_2l}{(\chi_2l-b_2)_+}\ ,\ \frac{\chi_1l}{(\chi_1l-a_2)_+}\ ,\ \frac{\chi_2k}{(\chi_2k-b_1)_+}\right\}. $$
For every $1\leq \gamma<\gamma^*$, we have that
\begin{equation}\label{ra-005}
 a_1>\frac{\chi_1k(\gamma-1)}{\gamma}\ ,\ b_1> \frac{\chi_2k(\gamma-1)}{\gamma}\ ,\quad b_2>\frac{\chi_2l(\gamma-1)}{\gamma}\ ,\ \text{and}\  a_2>\frac{\chi_1l(\gamma-1)}{\gamma}.
\end{equation}
It follows from \eqref{zz001}, \eqref{zz002}, \eqref{zz003}, and \eqref{ra-005} that
\begin{align*}
\frac{1}{\gamma}\frac{d}{dt}\int_{\Omega}u^{\gamma}(t)&\leq -\left(a_1-\frac{\chi_1k(\gamma-1)}{\gamma}\right)\int_{\Omega}u^{\gamma+1}(t)-(a_2-\frac{\chi_1l(\gamma-1)}{\gamma})\int_{\Omega}u^{\gamma}(t)v(t)+C\int_{\Omega}u^{\gamma}(t)\nonumber\\
&\leq -\left(a_1-\frac{\chi_1k(\gamma-1)}{\gamma}\right)\int_{\Omega}u^{\gamma+1}(t)+C\int_{\Omega}u^{\gamma}(t)\nonumber\\
&\leq -\left(a_1-\frac{\chi_1k(\gamma-1)}{\gamma}\right)|\Omega|^{-\frac{1}{\gamma}}\left(\int_{\Omega}u^{\gamma}(t)\right)^{\frac{\gamma+1}{\gamma}}+C\int_{\Omega}u^{\gamma}(t).
\end{align*}
Thus, it follows from Comparison principle for ODE's that
\begin{equation}\label{ra-006}
C_{\gamma}(u):=\sup_{0\leq t<T_{max}}\int_{\Omega}u^{\gamma}(t)<\infty.
\end{equation}
Recall that for every $0\leq t<T_{\max}$, we have
\begin{align*}
\frac{1}{\gamma}\frac{d}{dt}\int_{\Omega}v^{\gamma}(t)&\leq -\left(b_2-\frac{\chi_2l(\gamma-1)}{\gamma}\right)\int_{\Omega}v^{\gamma+1}(t)-(b_1-\frac{\chi_2k(\gamma-1)}
{\gamma})\int_{\Omega}v^{\gamma}(t)u(t)+C\int_{\Omega}v^{\gamma}(t).
\end{align*}
Hence, similar arguments as above yield that
\begin{equation}\label{ra-007}
C_{\gamma}(v):=\sup_{0\leq t<T_{max}}\int_{\Omega}v^{\gamma}(t)<\infty.
\end{equation}
So the result follows.

\noindent (3) \&(4) Set $\bar{\gamma}=\max\{1,\frac{n}{2}\}$ . Let us define $\gamma_{\infty}$ by
$$
\gamma_{\infty}=\sup\{\gamma_0>1\ \text{such that }\ \eqref{restate-eq1}\ holds\}.
$$

\noindent If $\min\{f(\bar{\gamma}),g(\bar{\gamma})\}>0$, then Lemma \ref{lem-0004}(1) implies that $\gamma_\infty>\bar{\gamma}$.\\
If $0<\min\left\{a_1,b_1,a_2,b_2,\frac{\chi_1k}{(\chi_1k-a_1)_+}-\frac{n}{2}\ ,\ \frac{\chi_2l}{(\chi_2l-b_2)_+}-\frac{n}{2}\ ,\ \frac{\chi_1l}{(\chi_1l-a_2)_+}-\frac{n}{2}\ ,\ \frac{\chi_2k}{(\chi_2k-b_1)_+}-\frac{n}{2}\right\}$, then Lemma \ref{lem-0004}(2) implies that $\gamma_\infty>\bar{\gamma}$.
\smallskip

\noindent {\bf Claim :} $\gamma_{\infty}=\infty.$

Suppose on the contrary that $\gamma_{\infty}<\infty$. Choose $\gamma_0$, with  $\max\{\bar{\gamma},\frac{\gamma_{\infty}}{2}\} <\gamma_0< \gamma_{\infty}$, be fixed. Next, choose $\gamma\in (\gamma_\infty\ ,\ 2\gamma_0)$. By definition of $\gamma_{\infty},$ we have that $\gamma_{0}$ satisfies \eqref{restate-eq1}. Note also that $\gamma_0$ that $\frac{2\gamma_0}{\gamma}>1.$ Using Gargliado-Nirenberg inequality, their is a constant $C_0>0$ such that
\begin{align}\label{ra-001}
\int_{\Omega}u^{1+\gamma}=\|u^{\frac{\gamma}{2}}\|_{\frac{2(1+\gamma)}{\gamma}}^{\frac{2(1+\gamma)}{\gamma}}& \leq C_0 \|\nabla u^{\frac{\gamma}{2}}\|_2^{\frac{2(1+\gamma)}{\gamma}\alpha} \|u^{\frac{\gamma}{2}}\|_{\frac{2\gamma_0}{\gamma}}^{\frac{2(1+\gamma)}{\gamma}(1-\alpha)}+C_0\|u^{\frac{\gamma}{2}}\|_{\frac{2}{\gamma}}^{\frac{2(1+\gamma)}{\gamma}}\nonumber\\
& = C_0 \|\nabla u^{\frac{\gamma}{2}}\|_2^{\frac{2(1+\gamma)}{\gamma}\alpha} \|u\|_{\gamma_0}^{(1+\gamma)(1-\alpha)}+C_0(\int_{\Omega}u)\left(\int_{\Omega}u\right)^{\gamma}\nonumber\\
& \leq  C_0 \|\nabla u^{\frac{\gamma}{2}}\|_2^{\frac{2(1+\gamma)}{\gamma}\alpha}\|u\|_{\gamma_0}^{(1+\gamma)(1-\alpha)}+C_0|\Omega|^{\gamma-1}(\int_{\Omega}u)\int_{\Omega}u^{\gamma},
\end{align}
where  $$ \alpha=\frac{\frac{\gamma}{2}[\frac{1}{\gamma_0}-\frac{1}{1+\gamma}]}{-\frac{1}{2}+\frac{1}{n}+\frac{\gamma}{2\gamma_0}}. $$ Since $\gamma_0$ satisfies \eqref{restate-eq1}, there is a constant $C_1>C_0$, independent of time,  such that inequality \eqref{ra-001} can be improved  to
\begin{equation}\label{ra-002}
\int_{\Omega}u^{1+\gamma}\leq C_1\|\nabla u^{\frac{\gamma}{2}}\|_{2}^{\frac{2(1+\gamma)}{\gamma}\alpha}+C_1\int_{\Omega}u^{\gamma}.
\end{equation}
Note that $\frac{\gamma}{(\gamma+1)\alpha}>1$ since
$\gamma_0>\frac{n}{2}.$ Then, by
Young's Inequality, there is $\tilde{C}_1>C_1$  such that \eqref{ra-002} can be improved to
\begin{eqnarray}\label{ra-003}
\left(f(\gamma)-a_1\right)\int_{\Omega}u^{1+\gamma} \leq \frac{4(\gamma-1)d_1}{\gamma^2} \|\nabla u^{\frac{\gamma}{2}}\|_2^2+\tilde{C}_1+\tilde{C}_1\int_{\Omega}u^{\gamma}.
\end{eqnarray}
Similar arguments yield that there is some positive constant $\tilde{C}_2>0$ such that
\begin{eqnarray}\label{ra-004}
\left(g(\gamma)-b_2\right)\int_{\Omega}v^{1+\gamma} \leq \frac{4(\gamma-1)d_2}{\gamma^2} \|\nabla v^{\frac{\gamma}{2}}\|_2^2+\tilde{C}_2+\tilde{C}_2\int_{\Omega}v^{\gamma}.
\end{eqnarray}
Combining \eqref{L0001}, \eqref{ra-003} and \eqref{ra-004}, there a positive constant $\tilde{C}$ such that
$$
\frac{1}{\gamma}\frac{d}{dt}\left(\int_{\Omega}u^{\gamma} +\int_{\Omega}v^{\gamma} \right)\leq -a_1\int_{\Omega}u^{1+\gamma}-b_2\int_{\Omega}v^{1+\gamma} +\tilde{C}\left(\int_{\Omega}u^{\gamma} +\int_{\Omega}v^{\gamma} \right)+\tilde{C}.
$$
Since $a_1>0$ and $b_2>0$, it follows from comparison principle for ODE's and the last inequality that
$$
\sup_{0\leq t<T_{\max}}\left(\int_{\Omega}u^{\gamma} +\int_{\Omega}v^{\gamma} \right) <\infty.
$$
It thus follows from the last inequality
$$
\sup_{1\leq p\leq \gamma, \ 0\leq t<T_{\max}}\left(\int_{\Omega}u^{p} +\int_{\Omega}v^{p} \right)\leq \left(\sup_{1\leq p\leq \gamma}|\Omega|^{\frac{1}{p}-\frac{1}{\gamma}}\right)\sup_{0\leq t<T_{\max}}\left(\int_{\Omega}u^{\gamma} +\int_{\Omega}v^{\gamma} \right) <\infty.
$$
That  is $\gamma$ satisfies \eqref{restate-eq1}. This implies that $\gamma\leq \gamma_\infty$.  Which is impossible since $\gamma>\gamma_\infty$.  Hence $\gamma_\infty=\infty$.
\end{proof}
 A natural question to ask is under what condition would \eqref{restate-eq1} holds for every $\gamma_0>0$? The next corollary provide a sufficient condition for \eqref{restate-eq1} to be satisfied for every $\gamma_0>0$.

\begin{corollary}\label{lem-0005} Assume that {\bf (H3)} holds. If in addition, either $\min\{f(\frac{n}{2}),g(\frac{n}{2})\}>0$ or  {\bf (H4)}, or $0<\min\left\{a_1,b_1,a_2,b_2,\frac{\chi_1k}{(\chi_1k-a_1)_+}-\frac{n}{2}\ ,\ \frac{\chi_2l}{(\chi_2l-b_2)_+}-\frac{n}{2}\ ,\ \frac{\chi_1l}{(\chi_1l-a_2)_+}-\frac{n}{2}\ ,\ \frac{\chi_2k}{(\chi_2k-b_1)_+}-\frac{n}{2}\right\}$ holds,  then for every $\gamma_0\geq 1$ and any nonnegative initial functions $u_0,v_0\in C^0(\overline{\Omega})$, the classical solution $(u(\cdot,\cdot),v(\cdot,\cdot), w(\cdot,\cdot))$ of \eqref{u-v-w-eq1} with initial $(u(\cdot,0),v(\cdot,0))=(u_0,v_0)$ satisfies
$$
\sup_{1\leq \gamma\leq \gamma_0\ , 0\leq t<T_{max}}\left(\int_{\Omega}u^{\gamma}(t)+\int_{\Omega}v^{\gamma}(t)  \right)<\infty.
$$
\end{corollary}
\begin{proof}
Observe that
$$
\lim_{\gamma\to\infty}f(\gamma)=a_1-\left((a_{2})_- +\frac{(l+k)\chi_1}{d_3}\right)
$$
and
$$
\lim_{\gamma\to\infty}g(\gamma)=b_2-\left((b_{1})_- +\frac{(l+k)\chi_2}{d_3}\right).
$$
Note that if {\bf (H3)} holds, then $\min\{f(1),g(1)\}>0$ and by Lemma \ref{lem-0003} ,\eqref{L1-cond} holds.  If {\bf (H4)} holds, there is a sequence of positive real numbers $\{\bar\gamma_m\}_{m\geq 1}$ with $\lim_{m\to\infty}\bar\gamma_m=\infty$ such that
$$
\min\{f(\bar{\gamma}_m),g(\bar{\gamma}_m)\}>0,\quad \forall \ m\geq1.
$$
The result thus follows from Lemma \ref{lem-0004}.
\end{proof}
\begin{corollary}\label{lem-0006} Assume that {\bf (H2)} holds. If in addition,  $$0<\min\left\{a_1,b_1,a_2,b_2,\frac{\chi_1k}{(\chi_1k-a_1)_+}-\frac{n}{2}\ ,\ \frac{\chi_2l}{(\chi_2l-b_2)_+}-\frac{n}{2}\ ,\ \frac{\chi_1l}{(\chi_1l-a_2)_+}-\frac{n}{2}\ ,\ \frac{\chi_2k}{(\chi_2k-b_1)_+}-\frac{n}{2}\right\}$$ holds,  then for every $\gamma_0\geq 1$ and any nonnegative initial functions $u_0,v_0\in C^0(\overline{\Omega})$, the classical solution $(u(\cdot,\cdot),v(\cdot,\cdot), w(\cdot,\cdot))$ of \eqref{u-v-w-eq1} with initial $(u(\cdot,0),v(\cdot,0))=(u_0,v_0)$ satisfies
$$
\sup_{1\leq \gamma\leq \gamma_0\ , 0\leq t<T_{max}}\left(\int_{\Omega}u^{\gamma}(t)+\int_{\Omega}v^{\gamma}(t)  \right)<\infty.
$$
\end{corollary}
\begin{proof}
Note that if {\bf (H2)} holds, then  by Lemma \ref{lem-0002}, \eqref{L1-cond} holds.
The result thus follows from Lemma \ref{lem-0004} (4).
\end{proof}

Now, by using the previous lemmas, we prove Theorem \ref{thm-global-001}.

\vspace{-0.1in}\begin{proof} [Proof of Theorem \ref{thm-global-001}]

(1) Let $(\overline u(t),\underline u(t),\overline v(t),\underline v(t))$ be as in lemma \ref{lem-0001}.
 It suffices to prove that $0\leq \underline{u}(t) \leq u(x,t;0,u_0,v_0) \leq \overline{u} (t)$ and  $0\leq \underline{v}(t) \leq v(x,t;0,u_0,v_0) \leq \overline{v} (t)$ for all $ 0\le t<T_{\max}$ and $x \in \bar{\Omega}$.  This method is the so called  rectangles method. 

 Observe that for any $\epsilon>0$,      there exists $0<t_{\epsilon }<T_{\max}$ such that
\vspace{-0.1in}\begin{equation}
\label{cont-1}
\underline{u}(t)-2\epsilon <u(x,t;0,u_0,v_0) <\overline{u}(t)+2\epsilon, \quad \text{ for all $(x, t) \in \Omega \times [0, t_{\epsilon})$,}
\vspace{-0.1in}\end{equation}

\vspace{-0.1in}\begin{equation}
\label{cont-2}
\underline{v}(t)-2\epsilon <v(x,t;0,u_0,v_0) <\overline{v}(t)+2\epsilon, \quad \text{ for all $(x, t) \in \Omega \times [0, t_{\epsilon})$,}
\vspace{-0.1in}\end{equation}
and by  comparison principle for elliptic equations,
\vspace{-0.1in}\begin{equation}
\label{elliptic-reg}
k\underline{u}(t)+l\underline{v}(t)-2(k+l)\epsilon \leq \lambda w(x,t) \leq k\overline{u}(t)+l\overline{v}(t)+2(k+l)\epsilon, \quad \text{ for all $(x, t) \in \Omega \times [0, t_{\epsilon}).$}
\vspace{-0.1in}\end{equation}
Let
\vspace{-0.1in}$$T_\epsilon=\sup\{t_\epsilon\in (0,T_{\max})\, \text{such that }\,\, \eqref{cont-1}\, \text{and}\, \eqref{cont-2} \, \text{hold}\}.
\vspace{-0.05in}$$
It then suffices to prove  that $T_\epsilon= T_{\max}$.
Assume by contradiction that $T_\epsilon<T_{\max}$.
Then there is  $x_0 \in\bar\Omega$ such that
\vspace{-0.05in}$$
{\rm either}\,\,\, u(x_0, T_{\epsilon};0,u_0,v_0)=\underline{u}(T_{\epsilon })-2\epsilon\,\,\, {\rm    or}\,\,\, u(x_0, T_{\epsilon};0,u_0,v_0)=\overline{u}(T_{\epsilon })+2\epsilon,
$$
or there is $y_0\in\bar\Omega$ such that
\vspace{-0.05in}$$
{\rm either}\,\,\, v(y_0, T_{\epsilon};0,u_0,v_0)=\underline{v}(T_{\epsilon })-2\epsilon\,\,\, {\rm    or}\,\,\, v(y_0, T_{\epsilon};0,u_0,v_0)=\overline{v}(T_{\epsilon })+2\epsilon.
$$
Let $\overline{U}(x,t)=u(x,t;0,u_0,v_0)-\overline{u}(t),$  $\underline{U}(x,t)=u(x,t;0,u_0,v_0)-\underline{u}(t),$ $\overline{V}(x,t)=v(x,t;0,u_0,v_0)-\overline{v}(t)$ and $\underline{V}(x,t)=v(x,t;0,u_0,v_0)-\underline{v}(t).$ Note that  for $t\in (0,T_{\max})$, $\overline{U}$ satisfies
 \vspace{-0.05in}
 \begin{align*}
 \overline{U}_t-d_1\Delta \overline{U}= &-\chi_1  \nabla\overline{U} \cdot \nabla w
   +\overline{U} \left[a_0-\Big(a_1-\chi_1k\Big)(u+\overline{u})  -  \chi_1 \lambda w+\chi_1 l v\right] \\
   &+\underbrace{\frac{\chi_1l}{d_3}\overline{u}\overline{V}}_{I_{0}(x,t)} +\underbrace{\frac{\chi_1}{d_3} \overline u (-\lambda w+k\underline{u}+l\underline{v})}_{I_{1}(x,t)}
                                                                     +\underbrace{\Big(a_2\Big)_+(-uv+ \overline{u}\underline{v})}_{I_{2}(x,t)}+\underbrace{\Big(a_2\Big)_-(uv- \overline{u}\,\overline{v})}_{I_{3}(x,t)}\\
 &+\underbrace{\Big(a_3\Big)_+\left(- u(\int_{\Omega}u)+|\Omega|\overline u \underline u\right)}_{I_{4}(x,t)} +\underbrace{\Big(a_3\Big)_-\left( u(\int_{\Omega}u)-|\Omega|\overline {u}^2\right)}_{I_{5}(x,t)}\\
& +\underbrace{\Big(a_4\Big)_+\left(- u(\int_{\Omega}v)+|\Omega|\overline u \underline v\right)}_{I_{6}(x,t)}+\underbrace{\Big(a_4\Big)_-\left( u(\int_{\Omega}v)-|\Omega|\overline u\, \overline v\right)}_{I_{7}(x,t)}. \end{align*}

By multiplying the above inequality by $\overline{U}_{+}$ and integrating with respect to $x$ over $\Omega$, we get
\begin{eqnarray*}
  & & \frac{1}{2}\frac{d}{dt}\int_{\Omega}(\overline{U}_{+})^2       +d_1\int_{\Omega}|\nabla (\overline{U}_{+}) |^2\\
 &  & \le \int_{\Omega}(\overline{U}_{+})^2 \left[  a_0-\Big(a_1-\chi_1k\Big)(u+\overline{u})  -  \chi_1 \lambda w+\chi_1 l v+\frac{\chi_1(\lambda w-ku-lv)}{2d_3}  \right]                 \nonumber \\
&&\quad + \int_{\Omega}(\overline{U}_{+}){I_{0}(x,t)} dx   + \int_{\Omega}(\overline{U}_{+}){I_{1}(x,t)} dx + \int_{\Omega}(\overline{U}_{+}){I_{2}(x,t)} dx+ \int_{\Omega}(\overline{U}_{+}){I_{3}(x,t)}dx \nonumber\\
&&\quad + \int_{\Omega}(\overline{U}_{+}){I_{4}(x,t)}dx    + \int_{\Omega}(\overline{U}_{+}){I_{5}(x,t)}dx  + \int_{\Omega}(\overline{U}_{+}){I_{6}(x,t)} dx+ \int_{\Omega}(\overline{U}_{+}){I_{7}(x,t)}dx \nonumber
\end{eqnarray*}
for $t\in (0\ ,\ T_{\max})$.  For every $t\in(0\ ,\ T_{\varepsilon})$, we have
 \vspace{-0.1in}
 \begin{align*}
 \int_{\Omega}(\overline{U}_{+}){I_{0}(x,t)}=\frac{\chi_1l}{d_3}\overline{u} \int_{\Omega}(\overline{U}_{+})\overline{V}                                                                           &\leq\frac{\chi_1l}{d_3}\overline{u} \int_{\Omega}(\overline{U}_{+})(\overline{V}_+)\leq \frac{\chi_1l{\overline{u}}}{2d_3}\Big(\int_{\Omega}(\overline{U}_{+})^2+ \int_{\Omega}(\overline{V}_+)^2\Big),
                                                                         \end{align*}
and
 \begin{align*}
 \int_{\Omega}(\overline{U}_{+}){I_{1}(x,t)}
&\leq \frac{\chi_1\overline{u}}{2d_3}\Big(\int_{\Omega}(\overline{U}_{+})^2+ \int_{\Omega}\left((\lambda w-k\underline{u}-l\underline{v})_-\right)^2\Big).\end{align*}
Moreover by  using the third equation of  $\eqref{u-v-w-eq1},$  we get
\vspace{-0.1in}\begin{align*}  &\frac{d_3}{\lambda}\int_{\Omega}| \nabla(\lambda w-k\underline{u}-l\underline{v})_-|^2 +\int_{\Omega}\left((\lambda w
-k\underline{u}-l\underline{v})_-\right)^2 \\
& =  -k \int_{\Omega}(\underline{U}) (\lambda w-k\underline{u}-l\underline{v})_- -l \int_{\Omega}(\underline{V}) (\lambda w-k\underline{u}-l\underline{v})_-\\
& \leq k \int_{\Omega}(\underline{U})_- (\lambda w-k\underline{u}-l\underline{v})_-  +l \int_{\Omega}(\underline{V})_- (\lambda w-k\underline{u}-l\underline{v})_- \\
&\leq k^2\int_{\Omega}(\underline{U}_{-})^2+l^2\int_{\Omega}(\underline{V}_{-})^2+\frac{1}{2}\int_{\Omega}\left((\lambda w-k\underline{u}-l\underline{v})_-\right)^2.\end{align*}
Therefore \[\int_{\Omega}\left((\lambda w-k\underline{u}-l\underline{v})_-\right)^2\leq  2k^2\int_{\Omega}(\underline{U}_{-})^2+2l^2\int_{\Omega}(\underline{V}_{-})^2. \]
Thus
\[ \int_{\Omega}(\overline{U}_{+}){I_{1}(x,t)}\leq \frac{\chi_1\overline{u}}{2d_3}\Big(\int_{\Omega}(\overline{U}_{+})^2+ 2k^2\int_{\Omega}
(\underline{U}_{-})^2+2l^2\int_{\Omega}(\underline{V}_{-})^2\Big),\]
\vspace{-0.1in}
\begin{align*}
\int_{\Omega}(\overline{U}_{+}){I_{2}(x,t)} dx 
=-\Big(a_2\Big)_+ \int_{\Omega}(\overline{U}_{+})\left(\overline{U}v+\overline{u}\underline{V}\right)
& =-\Big(a_2\Big)_+ \int_{\Omega}(\overline{U}_{+})^2v-\Big(a_2\Big)_+ \overline{u}\int_{\Omega}(\overline{U}_{+})\underline{V}\\
&\leq \Big(a_2\Big)_+ \overline{u}\int_{\Omega}(\overline{U}_{+})(\underline{V}_-)\\
&\leq \frac{\Big(a_2\Big)_+
\overline{u}}{2}\Big(\int_{\Omega}(\overline{U}_{+})^2+
\int_{\Omega}(\underline{V}_-)^2\Big),\end{align*}
\vspace{-0.1in}
\begin{align*}
 \int_{\Omega}(\overline{U}_{+}){I_{3}(x,t)}dx
& =\Big(a_2\Big)_- \int_{\Omega}(\overline{U}_{+})u\overline{V}+\Big(a_2\Big)_- \overline{v}\int_{\Omega}(\overline{U}_{+})^2\\
&\leq\Big(a_2\Big)_-(\overline u +2\epsilon) \int_{\Omega}(\overline{U}_{+})(\overline{V}_+)+\Big(a_2\Big)_- \overline{v}\int_{\Omega}(\overline{U}_{+})^2 \\
&\leq \frac{\Big(a_2\Big)_{-} (\overline u
+2\epsilon)}{2}\Big(\int_{\Omega}(\overline{U}_{+})^2+
\int_{\Omega}(\overline{V}_+)^2\Big) +\Big(a_2\Big)_-
\overline{v}\int_{\Omega}(\overline{U}_{+})^2,\end{align*}
\vspace{-0.1in}
\begin{align*}
 \int_{\Omega}(\overline{U}_{+}){I_{4}(x,t)}dx
& =-\Big(a_3\Big)_+ \int_{\Omega}(\overline{U}_{+})u\int_{\Omega}\underline{U}-|\Omega|\underline{u}\Big(a_3\Big)_+ \int_{\Omega}(\overline{U}_{+})^2\\
&\leq\Big(a_3\Big)_+ \int_{\Omega}(\overline{U}_{+})u\int_{\Omega}(\underline{U}_-)\\
&\leq \frac{\Big(a_3\Big)_+ (\overline u
+2\epsilon)}{2}|\Omega|\Big(\int_{\Omega}(\overline{U}_{+})^2+
\int_{\Omega}(\underline{U}_-)^2\Big) ,\end{align*}
\vspace{-0.1in}
\begin{align*}
 \int_{\Omega}(\overline{U}_{+}){I_{5}(x,t)}dx
 &=\Big(a_3\Big)_{-} \int_{\Omega}(\overline{U}_{+})\left(u\int_{\Omega}\overline{U}+|\Omega|\overline{u}\overline{U}\right) \\
&\leq\Big(a_3\Big)_{-}(\overline u +2\epsilon)\int_{\Omega}(\overline{U}_{+})\int_{\Omega}(\overline{U}_{+})+|\Omega|\overline{u}\Big(a_3\Big)_{\color{red}-} \int_{\Omega}(\overline{U}_{+})^2 \\
&\leq |\Omega|\Big(a_3\Big)_{-}\left[ (\overline u
+2\epsilon)+\overline{u}\right]
\int_{\Omega}(\overline{U}_{+})^2,\end{align*}
\vspace{-0.1in}
\begin{align*}
 \int_{\Omega}(\overline{U}_{+}){I_{6}(x,t)}dx
 &=-\Big(a_4\Big)_+ \int_{\Omega}(\overline{U}_{+})\left( u\int_{\Omega}\underline{V}+|\Omega|\underline{v}\overline{U}\right) \\
&\leq\Big(a_4\Big)_+(\overline u +2\epsilon) \int_{\Omega}(\overline{U}_{+})\int_{\Omega}(\underline{V}_-)\\
&\leq |\Omega|\frac{\Big(a_4\Big)_+(\overline u +2\epsilon)
}{2}\Big(\int_{\Omega}(\overline{U}_{+})^2+
\int_{\Omega}(\underline{V}_-)^2\Big),\end{align*} and
\vspace{-0.1in}
\begin{align*}
 \int_{\Omega}(\overline{U}_{+}){I_{7}(x,t)}dx
 &=\Big(a_4\Big)_- \int_{\Omega}(\overline{U}_{+})\left( u\int_{\Omega}\overline{V}+|\Omega|\overline{v}\overline{U}\right) \\
&\leq\Big(a_4\Big)_-(\overline u +2\epsilon) \int_{\Omega}(\overline{U}_{+})\int_{\Omega}(\overline{V}_+)+|\Omega|\overline{v}\Big(a_4\Big)_- \int_{\Omega}(\overline{U}_{+})^2\\
&\leq |\Omega|\frac{\Big(a_4\Big)_-(\overline u +2\epsilon)  }{2}\Big(\int_{\Omega}(\overline{U}_{+})^2+ \int_{\Omega}(\overline{V}_+)^2\Big)+|\Omega|\overline{v}
\Big(a_4\Big)_- \int_{\Omega}(\overline{U}_{+})^2.\end{align*}
By combining all these inequalities, there is a constant $C_1=C_1(a_i,b_i,k,l,\chi_1,|\Omega|)$ such that
\[\frac{d}{dt}\int_{\Omega}(\overline{U}_{+})^2\leq C_1\Big(\int_{\Omega}(\overline{U}_{+})^2+\int_{\Omega}(\overline{V}_{+})^2+\int_{\Omega}(\underline{U}_{-})^2+\int_{\Omega}(\underline{V}_{-})^2\Big)\quad {\rm for}\quad t\in (0, T_\epsilon].\]
In a similar way, we get:
\[\frac{d}{dt}\int_{\Omega}(\overline{ V}_{+})^2\leq C_2\Big(\int_{\Omega}(\overline{U}_{+})^2+\int_{\Omega}(\overline{V}_{+})^2+\int_{\Omega}(\underline{U}_{-})^2+\int_{\Omega}(\underline{V}_{-})^2\Big)\quad {\rm for}\quad t\in (0, T_\epsilon],\]
\[\frac{d}{dt}\int_{\Omega}(\underline{U}_{-})^2\leq C_3\Big(\int_{\Omega}(\overline{U}_{+})^2+\int_{\Omega}(\overline{V}_{+})^2+\int_{\Omega}(\underline{U}_{-})^2+\int_{\Omega}(\underline{V}_{-})^2\Big)\quad {\rm for}\quad t\in (0, T_\epsilon],\]
and
\[\frac{d}{dt}\int_{\Omega}(\underline{v}_{-})^2\leq C_4\Big(\int_{\Omega}(\overline{U}_{+})^2+\int_{\Omega}(\overline{V}_{+})^2+\int_{\Omega}(\underline{U}_{-})^2+\int_{\Omega}(\underline{V}_{-})^2\Big)\quad {\rm for}\quad t\in (0, T_\epsilon].\]
Therefore there is an positive constant $C=C(a_i,b_i,k,l,\chi_1,|\Omega|)$ such that
\begin{align}
\label{eq-uper-lower-odes}
&\frac{d}{dt}\Big(\int_{\Omega}(\overline{U}_{+})^2+\int_{\Omega}(\overline{V}_{+})^2+\int_{\Omega}(\underline{U}_{-})^2+\int_{\Omega}(\underline{V}_{-})^2\Big)\nonumber\\
&\leq
C\Big(\int_{\Omega}(\overline{U}_{+})^2+\int_{\Omega}(\overline{V}_{+})^2+\int_{\Omega}(\underline{U}_{-})^2+\int_{\Omega}(\underline{V}_{-})^2\Big)\quad
{\rm for}\quad t\in (0, T_\epsilon].
\end{align}

Since $\overline{U}_{+}(\cdot,0)=\underline{U}_-(\cdot,0)=\overline{V}_{+}(\cdot,t_0)=\underline{V}_-(\cdot,0)=0,$ $\eqref{eq-uper-lower-odes}$  implies  $\overline{U}_{+}(x,t)=\underline{U}_- (x,t)=\overline{V}_{+}(x,t)=\underline{V}_{+}(x,t)=0$ for  $(x,t)\in\Omega\times [0,T_{\epsilon}]$.
Therefore,
\vspace{-0.1in}$$\underline{u}(t) \leq u(x,t;0,u_0,v_0) \leq \overline{u} (t)\quad (x, t) \in \overline{\Omega} \times [0, T_{\epsilon}].
$$
and
\vspace{-0.1in}$$\underline{v}(t) \leq v(x,t;0,u_0,v_0) \leq \overline{v} (t)\quad (x, t) \in \overline{\Omega} \times [0, T_{\epsilon}].
$$
This is a contradiction. Therefore,
$T_{\epsilon}=T_{\max}$ and the result follows by lemma \ref{lem-0001} .

\medskip

(2) By Corollary \ref{lem-0006} we have that for any $\gamma>1$
$$
0 \leq \int_{\Omega}u^{\gamma}+\int_{\Omega}v^{\gamma} \leq C_\gamma
$$
and by standard arguments involving Moser Alikakos iteration method
or as in \cite{TW07,ITBWS16} we get $$\sup_{0\leq t<T_{max}}\|u(t)\|_{\infty} <\infty \quad \text{and}\quad \sup_{0\leq t<T_{max}}\|v(t)\|_{\infty}< \infty.$$

\medskip

(3) By Corollary \ref{lem-0005} we have that for any $\gamma>1$
$$
0 \leq \int_{\Omega}u^{\gamma}+\int_{\Omega}v^{\gamma} \leq C_\gamma
$$
and by standard arguments involving Moser Alikakos iteration method
or as in \cite{TW07,ITBWS16} we get $$\sup_{0\leq t<T_{max}}\|u(t)\|_{\infty} <\infty \quad \text{and}\quad \sup_{0\leq t<T_{max}}\|v(t)\|_{\infty}< \infty.$$
\end{proof}

\section{Asymptotic Behavior}
In this section, we study the asymptotic behavior of classical solutions of \eqref{u-v-w-eq1}. Throughout this section we shall suppose that the condition   {\bf (H1)} holds. Thus, under these conditions, Theorem \ref{thm-global-001}(1) implies that for every nonnegative initial $u_0,v_0\in C(\overline{\Omega})$, the classical solution $(u(\cdot,\cdot;u_0,v_0),v(\cdot,\cdot;u_0,v_0),w(\cdot,\cdot;u_0,v_0))$  is globally defined in time and bounded.  Next, for every nonnegative initial functions  $u_{0},v_{0}\in C(\overline{\Omega})$, we define  $$\overline{u}(u_0,v_0)=\limsup_{t \to \infty}(\max_{x \in \bar{\Omega}}u(x,t;u_0;v_0)),$$ $$\underline{u}(u_0,v_0)=\liminf_{t \to \infty}(\min_{x \in \bar{\Omega}}u(x,t;u_0;v_0)),$$  $$\overline{v}(u_0,v_0)=\limsup_{t \to \infty}(\max_{x \in \bar{\Omega}}v(x,t;u_0;v_0)),$$ and $$\underline{v}(u_0,v_0)=\liminf_{t \to \infty}(\min_{x \in \bar{\Omega}}v(x,t;u_0;v_0)).$$
Thus, from Theorem \ref{thm-global-001}(1), for every nonnegative initial $u_0,v_0\in C(\overline{\Omega})$,   we have
\[0\leq u(x,t) \leq \max\left\{\|u_0\|_\infty,\frac{a_0+\sqrt{a_0^2+4\big(a_1-k\frac{\chi_1}{d_3}-|\Omega|(a_3)_-\big)\big((a_2)_-+|\Omega|(a_4)_-+l\frac{\chi_1}{d_3}\big)M}}{2\big(a_1-k\frac{\chi_1}{d_3}-|\Omega|(a_3)_-\big)} \right\}, \]

and

\[0\leq v(x,t)\leq  \max\left\{\|v_0\|_\infty,\frac{b_0+\sqrt{b_0^2+4\big(b_2-l\frac{\chi_2}{d_3}-|\Omega|(b_4)_-\big)\big((b_1)_-+|\Omega|(b_3)_-+k\frac{\chi_2}{d_3}\big)M}}{2\big(b_2-l\frac{\chi_2}{d_3}-|\Omega|(b_4)_-\big)} \right\},\]
with
\[M=e^{\max\left\{\ln{\left(\|u_0\|_\infty\|v_0\|_\infty\right)},\,2\ln\{\frac{a_0+b_0}{2 L}\}\right\}},\]
where
\begin{align*}
L=\min\Big\{&a_1-k\frac{\chi_1+\chi_2}{d_3}-|\Omega|((a_3)_-+|(b_3)_-)-(b_1)_-,\\
&b_2-l\frac{\chi_2+\chi_1}{d_3}-|\Omega|((b_4)_-+(a_4)_-)-(a_2)_-\Big\}.
\end{align*}
Using the definition of $\limsup$ and  of $\liminf,$ and elliptic regularity,
we get that given $\epsilon >0,$ there exists $T_{\epsilon}>0$ such that
\begin{equation}
\label{Asym-eq-01}
\underline{u}(u_0;v_0)-\epsilon \leq u(x,t) \leq \overline{u}(u_0;v_0)+\epsilon , \quad \underline{v}(u_0;v_0)-\epsilon \leq v(x,t) \leq \overline{v}(u_0;v_0)+\epsilon,\quad \forall t>T_{\epsilon}
\end{equation}
and
\begin{equation}
\label{Asym-eq-02}
k\underline{u}(u_0;v_0)+l\underline{v}(u_0,v_0)-(l+k)\epsilon \leq \lambda w(x,t) \leq k\overline{u}+l\overline{v}+
(k+l)\epsilon ,\quad \forall t>T_{\epsilon}
\end{equation}
In what follows, we drop the dependence of $\overline{u},\underline{u},\overline{v}$ and $\underline{v}$ on $(u_0,v_0)$.

\subsection{Coexistence}
In this subsection, our aim is to find conditions on the parameters only which guarantee that  $$0<\underline{u}(u_0,v_0)=\overline{u}(u_0,v_0)=\frac{a_0(b_2+|\Omega|b_4)-b_0(a_2+|\Omega|a_4)}{(b_2+|\Omega|b_4)(a_1+|\Omega|a_3)-(a_2+|\Omega|a_4)(b_1+|\Omega|b_3)},$$ and  $$0<\underline{v}(u_0,v_0)=\overline{v}(u_0,v_0)=\frac{a_0(b_1+|\Omega|b_3)-b_0(a_1+|\Omega|a_3)}{(b_1+|\Omega|b_3)(a_2+|\Omega|a_4)-(a_1+|\Omega|a_3)(b_2+|\Omega|b_4)}.$$  This method is the so called eventual comparison method. 

\medskip

 Let $u_{0},v_0\in C(\overline{\Omega})$ be given  nonnegative initials such that $0<\min\{\|u_{0}\|_{\infty};\|v_0\|_{\infty}\}$. Observe that if either $\|u_{0}\|_{\infty}=0$ or $\|v_0\|_{\infty}=0$, system \eqref{u-v-w-eq1} reduces to the one species case and we refer the reader to \cite{ITBWS16}, \cite{TW07} and  references therein. Since $0<\min\{\|u_{0}\|_{\infty};\|v_0\|_{\infty}\}$, the maximum principle for parabolic equations implies that
 $$0<\min\{\|u(\cdot,t;u_0;v_0)\|_{\infty}; \|v(\cdot,t;u_0,v_0)\|_{\infty}\},\quad \forall\ t\geq 0. $$
Next, we prove the following two important lemmas toward the proof of the coexistence .
\begin{lemma}
\label{lem-asym-coexist-01}
Suppose {\bf (H1)} holds. Then
\begin{equation}\label{Asymp-coexist-eq-5}
\overline{u}\leq \frac{\left\{ a_0-\left(|\Omega|(a_3)_++\frac{\chi_1}{d_3}k\right)\underline{u}+|\Omega|(a_3)_-\overline{u}M_1(\overline{v},\underline{v})\right\}_{+}}{a_{1}-\frac{\chi_{1}k}{d_{3}}},
\end{equation}
and
\begin{equation}\label{Asymp-coexist-eq-6}
\underline{u}\geq \frac{\left\{ a_0-\left(|\Omega|(a_3)_++k\frac{\chi_1}{d_3}\right)\overline{u}+|\Omega|(a_3)_-\underline{u}M_2(\overline{v},\underline{v})\right\}_{+}}{a_{1}-\frac{\chi_{1}k}{d_{3}}},
\end{equation}
where
\[M_1(\overline{v},\underline{v})=-\left((a_2)_+ +|\Omega|(a_4)_++l\frac{\chi_{1}}{d_{3}}\right)\underline{v}+\left((a_2)_- +|\Omega|(a_4)_-+l\frac{\chi_{1}}{d_{3}}\right)\overline{v},\]
and
\[M_2(\overline{v},\underline{v})=-\left((a_2)_+ +|\Omega|(a_4)_++l\frac{\chi_{1}}{d_{3}}\right)\overline{v}+\left((a_2)_- +|\Omega|(a_4)_-+l\frac{\chi_{1}}{d_{3}}\right)\underline{v}.\]
\end{lemma}

\begin{proof}
Since {\bf (H1)} holds, we have : For every $t\geq T_{\varepsilon}$, it follows from \eqref{Asym-eq-01} and \eqref{Asym-eq-02} that
\begin{align}\label{F00}
&u_t-d_1\Delta u+\chi_1 \nabla u \cdot \nabla w\nonumber\\
&=u\left\{ a_0-(a_1-\frac{\chi_1 }{d_3}k)u -(a_2)_+v+((a_2)_-+l\frac{\chi_1 }{d_3})v-(a_3)_+\int_{\Omega}u+(a_3)_-\int_{\Omega}u\right\}\nonumber\\
&+u\left\{-(a_4)_+\int_{\Omega}v+(a_4)_-\int_{\Omega}v-\frac{\chi_1 }{d_3}\lambda w\right\}\nonumber\\
& \leq u\left\{ a_0-(a_1-\frac{\chi_1 }{d_3}k)u-\left(|\Omega|(a_3)_++k\frac{\chi_1}{d_3}\right)\underline{u}+|\Omega|(a_3)_-\overline{u}-\left((a_2)_+ +|\Omega|(a_4)_++l\frac{\chi_{1}}{d_{3}}\right)\underline{v}\right\} \nonumber\\
&+u\left\{ \left((a_2)_- +|\Omega|(a_4)_-+l\frac{\chi_{1}}{d_{3}}\right)\overline{v}+\left(|a_2|+|\Omega|(|a_3|+|a_4|)+k\frac{\chi_1}{d_3}+2l\frac{\chi_{1}}{d_{3}}\right)\epsilon\right\}.
\end{align}
Let $\overline{U}_{\varepsilon}$ be the solution of the following ODE
\begin{small}
 \begin{equation}\label{F01}
\begin{cases}
\overline{U}_{\varepsilon}'=\overline{U}_{\varepsilon}\left\{ a_0-(a_1-\frac{\chi_1 }{d_3}k)\overline{U}_{\varepsilon}-\left(|\Omega|(a_3)_++k\frac{\chi_1}{d_3}\right)\underline{u}+|\Omega|(a_3)_-\overline{u}-\left((a_2)_+ +|\Omega|(a_4)_++l\frac{\chi_{1}}{d_{3}}\right)\underline{v}\right\} \cr
\qquad+\overline{U}_{\varepsilon}\left\{ \left((a_2)_- +|\Omega|(a_4)_-+l\frac{\chi_{1}}{d_{3}}\right)\overline{v}+\left(|a_2|+|\Omega|(|a_3|+|a_4|)+k\frac{\chi_1}{d_3}+2l\frac{\chi_{1}}{d_{3}}\right)\epsilon\right\}\cr
\overline{U}( T_{\epsilon})=\|u(\cdot,T_{\varepsilon})\|_{\infty}\cr
\end{cases}
\end{equation}
\end{small}
Thus \eqref{F00}, \eqref{F01} and comparison principle for parabolic equations imply that
\begin{equation}\label{F02}
u(\cdot,t)\leq \overline{U}_{\varepsilon}(t),\quad \forall\ t\geq T_{\varepsilon}.
\end{equation}
Let $h=|a_2|+|\Omega|(|a_3|+|a_4|)+k\frac{\chi_1}{d_3}+2l\frac{\chi_{1}}{d_{3}}.$
Now, since $a_1-\frac{\chi_1 }{d_3}k>0$ by {\bf (H1)}, the function $\overline{U}_{\varepsilon}$ is globally defined in time and satisfies
$$\lim_{t\to \infty}\overline{U}_{\varepsilon}(t)=\frac{\left\{ a_0-(|\Omega|(a_3)_++\frac{\chi_1}{d_3}k)\underline{u}+|\Omega|(a_3)_-\overline{u}+M_1(\overline{v},\underline{v})+h\epsilon\right\}_{+}}{a_{1}-\frac{\chi_{1}k}{d_{3}}}.$$
Combining this with inequality \eqref{F02}, we obtain that
\begin{equation}\label{F001}
\overline{u}\leq\frac{\left\{ a_0-(|\Omega|(a_3)_++\frac{\chi_1}{d_3}k)\underline{u}+|\Omega|(a_3)_-\overline{u}+M_1(\overline{v},\underline{v})+h\epsilon\right\}_{+}}{a_{1}-\frac{\chi_{1}k}{d_{3}}}.
\end{equation}
Letting $\varepsilon$ tends to 0 in the last inequality, we obtain that
$$
\overline{u}\leq \frac{\left\{ a_0-\left(|\Omega|(a_3)_++\frac{\chi_1}{d_3}k\right)\underline{u}+|\Omega|(a_3)_-\overline{u}+M_1(\overline{v},\underline{v})\right\}_{+}}{a_{1}-\frac{\chi_{1}k}{d_{3}}}.
$$
Thus, \eqref{Asymp-coexist-eq-5} follows.

Similarly, for every $t\geq T_{\varepsilon}$, it follows from \eqref{Asym-eq-01} and \eqref{Asym-eq-02} that
\begin{align}\label{F03}
&u_t-d_1\Delta u+\chi_1 \nabla u \cdot \nabla w\nonumber\\
&=u\left\{ a_0-(a_1-\frac{\chi_1 }{d_3}k)u -(a_2)_+v+((a_2)_-+l\frac{\chi_1 }{d_3})v-(a_3)_+\int_{\Omega}u+(a_3)_-\int_{\Omega}u\right\}\nonumber\\
&+u\left\{-(a_4)_+\int_{\Omega}v+(a_4)_-\int_{\Omega}v-\frac{\chi_1 }{d_3}\lambda w\right\}\nonumber\\
& \geq u\left\{ a_0-(a_1-\frac{\chi_1 }{d_3}k)u-\left(|\Omega|(a_3)_++k\frac{\chi_1}{d_3}\right)\overline{u}+|\Omega|(a_3)_-\underline{u}-\left((a_2)_+ +|\Omega|(a_4)_++l\frac{\chi_{1}}{d_{3}}\right)\overline{v}\right\} \nonumber\\
&+u\left\{ \left((a_2)_- +|\Omega|(a_4)_-+l\frac{\chi_{1}}{d_{3}}\right)\underline{v}-\left(|a_2|+|\Omega|(|a_3|+|a_4|)+k\frac{\chi_1}{d_3}+2l\frac{\chi_{1}}{d_{3}}\right)\epsilon\right\}.
\end{align}
Let $\underline{U}_{\varepsilon}$ be the solution of the following ODE
\begin{small}
 \begin{equation}\label{F04}
\begin{cases}
\underline{U}_{\varepsilon}'=\underline{U}_{\varepsilon}\left\{ a_0-(a_1-\frac{\chi_1 }{d_3}k)\underline{U}_{\varepsilon}-\left(|\Omega|(a_3)_++k\frac{\chi_1}{d_3}\right)\overline{u}+|\Omega|(a_3)_-\underline{u}-\left((a_2)_+ +|\Omega|(a_4)_++l\frac{\chi_{1}}{d_{3}}\right)\overline{v}\right\} \cr
\qquad+\underline{U}_{\varepsilon}\left\{ \left((a_2)_- +|\Omega|(a_4)_-+l\frac{\chi_{1}}{d_{3}}\right)\underline{v}-\left(|a_2|+|\Omega|(|a_3|+|a_4|)+k\frac{\chi_1}{d_3}+2l\frac{\chi_{1}}{d_{3}}\right)\epsilon\right\}\cr
\underline{U}( T_{\epsilon})=\min_{x \in \bar{\Omega}}u(x,T_{\varepsilon})\cr
\end{cases}
\end{equation}
\end{small}
Thus \eqref{F03}, \eqref{F04} and comparison principle for parabolic equations imply that
\begin{equation}\label{F05}
u(\cdot,t)\geq \overline{U}_{\varepsilon}(t),\quad \forall\ t\geq T_{\varepsilon}.
\end{equation}
Now, since $a_1-\frac{\chi_1 }{d_3}k>0$ by {\bf (H1)}, the function $\underline{U}_{\varepsilon}$ is globally defined in time and satisfies
$$\lim_{t\to \infty}\underline{U}_{\varepsilon}(t)=\frac{\left\{ a_0-(|\Omega|(a_3)_++\frac{\chi_1}{d_3}k)\overline{u}+|\Omega|(a_3)_-\underline{u}+M_2(\overline{v},\underline{v})-h\epsilon\right\}_{+}}{a_{1}-\frac{\chi_{1}k}{d_{3}}}.$$
Combining this with inequality \eqref{F02}, we obtain that
\begin{equation}\label{F002}
\underline{u}\geq\frac{\left\{ a_0-(|\Omega|(a_3)_++\frac{\chi_1}{d_3}k)\overline{u}+|\Omega|(a_3)_-\underline{u}+M_2(\overline{v},\underline{v})-h\epsilon\right\}_{+}}{a_{1}-\frac{\chi_{1}k}{d_{3}}}.
\end{equation}
Letting $\varepsilon$ tends to 0 in the last inequality, we obtain that
$$
\underline{u}\geq \frac{\left\{ a_0-\left(|\Omega|(a_3)_++k\frac{\chi_1}{d_3}\right)\overline{u}+|\Omega|(a_3)_-\underline{u}+M_1(\overline{v},\underline{v})\right\}_{+}}{a_{1}-\frac{\chi_{1}k}{d_{3}}}.
$$
Thus, \eqref{Asymp-coexist-eq-6} follows.
\end{proof}

\begin{lemma}
\label{lem-asym-coexist-02}
Suppose {\bf (H1)} holds. Then
\begin{equation}\label{Asymp-coexist-eq-7}
\overline{v}\leq \frac{\left\{ b_0-\left(|\Omega|(b_4)_++\frac{\chi_2}{d_3}l\right)\underline{v}+|\Omega|(b_4)_-\overline{v}+M_1(\overline{u},\underline{u})\right\}_{+}}{b_{2}-\frac{\chi_{2}l}{d_{3}}},
\end{equation}
and
\begin{equation}\label{Asymp-coexist-eq-8}
\underline{v}\geq\frac{\left\{ b_0-\left(|\Omega|(b_4)_++\frac{\chi_2}{d_3}l\right)\overline{v}+|\Omega|(b_4)_-\underline{v}+M_2(\overline{u},\underline{u})\right\}_{+}}{b_{2}-\frac{\chi_{2}l}{d_{3}}},
\end{equation}
where
\[M_1(\overline{u},\underline{u})=-\left((b_1)_+ +|\Omega|(b_3)_++k\frac{\chi_{2}}{d_{3}}\right)\underline{u}+\left((b_1)_- +|\Omega|(b_3)_-+k\frac{\chi_{2}}{d_{3}}\right)\overline{u},\]
and
\[M_2(\overline{u},\underline{u})=-\left((b_1)_+ +|\Omega|(b_3)_++k\frac{\chi_{2}}{d_{3}}\right)\overline{u}+\left((b_1)_- +|\Omega|(b_3)_-+k\frac{\chi_{2}}{d_{3}}\right)\underline{u}.\]
\end{lemma}

\begin{proof}
Since {\bf (H1)} holds, we have :
\begin{align}\label{F06}
&v_t-d_2\Delta v+\chi_2 \nabla v \cdot \nabla w\nonumber\\
& \leq v\left\{ b_0-(b_2-\frac{\chi_2 }{d_3}k)v-\left(|\Omega|(b_4)_++l\frac{\chi_2}{d_3}\right)\underline{v}+|\Omega|(b_4)_-\overline{v}-\left((b_1)_+ +|\Omega|(b_3)_++k\frac{\chi_{2}}{d_{3}}\right)\underline{u}\right\} \nonumber\\
&+v\left\{ \left((b_1)_- +|\Omega|(b_3)_-+k\frac{\chi_{2}}{d_{3}}\right)\overline{u}+\left(|b_1|+|\Omega|(|b_3|+|b_4|)+l\frac{\chi_2}{d_3}+2k\frac{\chi_{1}}{d_{3}}\right)\epsilon\right\}.
\end{align}
Observe that inequality \eqref{F06} is similar to inequality \eqref{F00}. Thus, using $b_2-\frac{\chi_2 }{d_3}k>0$ by {\bf (H1)}, similar arguments used to establish \eqref{F001} yield that
$$
\overline{v}\leq \frac{\left\{ b_0-(|\Omega|(b_4)_++\frac{\chi_2}{d_3}l)\underline{v}+|\Omega|(b_4)_-\overline{v}+M_1(\overline{u},\underline{u})+h\epsilon\right\}_{+}}{b_{2}-\frac{\chi_{2}l}{d_{3}}}
$$
for every $\varepsilon>0$. Letting $\varepsilon$ tends to 0 in the last inequality, we obtain \eqref{Asymp-coexist-eq-7}.

Similarly, since {\bf (H1)} holds, we have :
\begin{align}\label{F07}
&v_t-d_2\Delta v+\chi_2 \nabla v \cdot \nabla w\nonumber\\
& \geq v\left\{ b_0-(b_2-\frac{\chi_2 }{d_3}k)v-\left(|\Omega|(b_4)_++l\frac{\chi_2}{d_3}\right)\overline{v}+|\Omega|(b_4)_-\underline{v}-\left((b_1)_+ +|\Omega|(b_3)_++k\frac{\chi_{2}}{d_{3}}\right)\overline{u}\right\} \nonumber\\
&+v\left\{ \left((b_1)_- +|\Omega|(b_3)_-+k\frac{\chi_{2}}{d_{3}}\right)\underline{u}-\left(|b_1|+|\Omega|(|b_3|+|b_4|)+l\frac{\chi_2}{d_3}+2k\frac{\chi_{1}}{d_{3}}\right)\epsilon\right\}.
\end{align}
Observe that inequality \eqref{F07} is similar to inequality \eqref{F002}. Thus, using $b_2-\frac{\chi_2 }{d_3}k>0$ by {\bf (H1)}, similar arguments used to establish \eqref{F002} yield that
$$
\underline{v}\geq \frac{\left\{ b_0-(|\Omega|(b_4)_++\frac{\chi_2}{d_3}l)\overline{v}+|\Omega|(b_4)_-\underline{v}+M_2(\overline{u},\underline{u})-h\epsilon\right\}_{+}}{b_{2}-\frac{\chi_{2}l}{d_{3}}}
$$
for every $\varepsilon>0$. Letting $\varepsilon$ tends to 0 in the last inequality, we obtain \eqref{Asymp-coexist-eq-8}.
\end{proof}

\begin{lemma}
\label{lem-asym-coexist-03}
Suppose {\bf (H1)} holds and equations \eqref{Asymp-coexist-eq-01}, \eqref{Asymp-coexist-eq-02}, and \eqref{Asymp-coexist-eq-03}. Then
\begin{align}
\label{Asymp-coexist-eq-9}
&a_0-\left(|\Omega|(a_3)_++\frac{\chi_1}{d_3}k\right)\underline{u}+|\Omega|(a_3)_-\overline{u}-\left((a_2)_+ +|\Omega|(a_4)_++l\frac{\chi_{1}}{d_{3}}\right)\underline{v} \nonumber\\
&+\left((a_2)_- +|\Omega|(a_4)_-+l\frac{\chi_{1}}{d_{3}}\right)\overline{v}\geq 0
\end{align}
and

\begin{align}\label{Asymp-coexist-eq-10}
& b_0-\left(|\Omega|(b_4)_++\frac{\chi_2}{d_3}l\right)\underline{v}+|\Omega|(b_4)_-\overline{v}-\left((b_1)_+ +|\Omega|(b_3)_++k\frac{\chi_{2}}{d_{3}}\right)\underline{u} \nonumber\\
&+\left((b_1)_- +|\Omega|(b_3)_-+k\frac{\chi_{2}}{d_{3}}\right)\overline{u}\geq 0.
 \end{align}

\end{lemma}

\begin{proof} Suppose by contradiction that
\begin{align*}
&a_0-\left(|\Omega|(a_3)_++\frac{\chi_1}{d_3}k\right)\underline{u}+|\Omega|(a_3)_-\overline{u}-\left((a_2)_+ +|\Omega|(a_4)_++l\frac{\chi_{1}}{d_{3}}\right)\underline{v} \nonumber\\
&+\left((a_2)_- +|\Omega|(a_4)_-+l\frac{\chi_{1}}{d_{3}}\right)\overline{v}< 0
\end{align*}
or
\begin{align*}
& b_0-\left(|\Omega|(b_4)_++\frac{\chi_2}{d_3}l\right)\underline{v}+|\Omega|(b_4)_-\overline{v}-\left((b_1)_+ +|\Omega|(b_3)_++k\frac{\chi_{2}}{d_{3}}\right)\underline{u} \nonumber\\
&+\left((b_1)_- +|\Omega|(b_3)_-+k\frac{\chi_{2}}{d_{3}}\right)\overline{u}< 0,
 \end{align*}
and the proof is divided into three cases

\noindent {\bf Case I.}
\begin{align*}
&a_0-\left(|\Omega|(a_3)_++\frac{\chi_1}{d_3}k\right)\underline{u}+|\Omega|(a_3)_-\overline{u}-\left((a_2)_+ +|\Omega|(a_4)_++l\frac{\chi_{1}}{d_{3}}\right)\underline{v} \nonumber\\
&+\left((a_2)_- +|\Omega|(a_4)_-+l\frac{\chi_{1}}{d_{3}}\right)\overline{v}< 0
\end{align*}
and
\begin{align*}
& b_0-\left(|\Omega|(b_4)_++\frac{\chi_2}{d_3}l\right)\underline{v}+|\Omega|(b_4)_-\overline{v}-\left((b_1)_+ +|\Omega|(b_3)_++k\frac{\chi_{2}}{d_{3}}\right)\underline{u} \nonumber\\
&+\left((b_1)_- +|\Omega|(b_3)_-+k\frac{\chi_{2}}{d_{3}}\right)\overline{u}< 0.
 \end{align*}
Then, it follows from Lemmas \ref{lem-asym-coexist-01} and \ref{lem-asym-coexist-02} that  $\underline{u}=\overline{u}=\underline{v}=\overline{v}=0$. Inserting these values in the last two inequalities, we obtain that $\max\{a_{0},b_{0}\}< 0$, which is a contradiction.

\medskip

\noindent {\bf Case II.}
\begin{align*}
&a_0-\left(|\Omega|(a_3)_++\frac{\chi_1}{d_3}k\right)\underline{u}+|\Omega|(a_3)_-\overline{u}-\left((a_2)_+ +|\Omega|(a_4)_++l\frac{\chi_{1}}{d_{3}}\right)\underline{v} \nonumber\\
&+\left((a_2)_- +|\Omega|(a_4)_-+l\frac{\chi_{1}}{d_{3}}\right)\overline{v}\geq  0
\end{align*}
and
\begin{align*}
& b_0-\left(|\Omega|(b_4)_++\frac{\chi_2}{d_3}l\right)\underline{v}+|\Omega|(b_4)_-\overline{v}-\left((b_1)_+ +|\Omega|(b_3)_++k\frac{\chi_{2}}{d_{3}}\right)\underline{u} \nonumber\\
&+\left((b_1)_- +|\Omega|(b_3)_-+k\frac{\chi_{2}}{d_{3}}\right)\overline{u}< 0.
 \end{align*}
Then by Lemma \ref{lem-asym-coexist-02}, we get  $\underline{v}=\overline{v}=0$. By Lemma \ref{lem-asym-coexist-01} we have that
$$
(\overline{u}-\underline{u})(a_1-\frac{\chi_1 }{d_3}k) \leq (\frac{\chi_1 }{d_3}k+|\Omega||a_3|)(\overline{u}-\underline{u}).
$$
Thus, inequality \eqref{Asymp-coexist-eq-01} implies that $\overline{u}=\underline{u}$. Next, solving for $\overline{u}$ in \eqref{Asymp-coexist-eq-5} and \eqref{Asymp-coexist-eq-6}, we obtain that \begin{equation}\label{F08}
\underline{u}=\overline{u} =\frac{a_0}{a_1+|\Omega|a_3}.
\end{equation}
 Finally combining inequality \eqref{Asymp-coexist-eq-8} with the fact that $\underline{v}=\overline{v}=0$, we obtain
$$
0\geq    b_0- (b_{1}+b_{3}|\Omega|)\overline{u},
$$
 which is equivalent to
$$
\underline{u}\geq \frac{b_{0}}{b_1+b_3|\Omega|}.
$$
The last inequality combined with \eqref{F08} yield that
$$
\frac{a_0}{a_1+|\Omega|a_3}\geq \frac{b_{0}}{b_1+b_3|\Omega|}.
$$
This contradicts inequality \eqref{Asymp-coexist-eq-03}.

\medskip

\noindent {\bf Case III.}
\begin{align*}
&a_0-\left(|\Omega|(a_3)_++\frac{\chi_1}{d_3}k\right)\underline{u}+|\Omega|(a_3)_-\overline{u}-\left((a_2)_+ +|\Omega|(a_4)_++l\frac{\chi_{1}}{d_{3}}\right)\underline{v} \nonumber\\
&+\left((a_2)_- +|\Omega|(a_4)_-+l\frac{\chi_{1}}{d_{3}}\right)\overline{v}<  0
\end{align*}
and
\begin{align*}
& b_0-\left(|\Omega|(b_4)_++\frac{\chi_2}{d_3}l\right)\underline{v}+|\Omega|(b_4)_-\overline{v}-\left((b_1)_+ +|\Omega|(b_3)_++k\frac{\chi_{2}}{d_{3}}\right)\underline{u} \nonumber\\
&+\left((b_1)_- +|\Omega|(b_3)_-+k\frac{\chi_{2}}{d_{3}}\right)\overline{u}\geq 0.
 \end{align*}
Then by Lemma \ref{lem-asym-coexist-01}, we get  $\underline{u}=\overline{u}=0$. By Lemma \ref{lem-asym-coexist-02} we have that
$$
(\overline{v}-\underline{v})(b_2-\frac{\chi_2 }{d_3}l) \leq (\frac{\chi_2 }{d_3}l+|\Omega||b_4|)(\overline{v}-\underline{v}).
$$
Thus, inequality \eqref{Asymp-coexist-eq-03} implies that  $\overline{v}=\underline{v}$. Next solving for $\overline{v}$ in \eqref{Asymp-coexist-eq-7} and \eqref{Asymp-coexist-eq-8}, we obtain that
\begin{equation}\label{F09}\underline{v}=\overline{v} =\frac{b_0}{b_2+|\Omega|b_4}.\end{equation}
 Finally, combining inequality \eqref{Asymp-coexist-eq-6} with the fact that $\underline{u}=\overline{u}=0$, we obtain
$$
0 \geq   a_0- (a_{2}+a_{4}|\Omega|)\overline{v}.
$$
Which is equivalent to
$$
\overline{v}\geq \frac{a_0}{a_{2}+a_4|\Omega|}.
$$
This contradicts inequality \eqref{Asymp-coexist-eq-03}.
\end{proof}

Now, based on these lemmas, we complete the proof of Theorem \ref{Asym-Th-1}.

\begin{proof}[Proof of Theorem \ref{Asym-Th-1}]
From lemmas \ref {lem-asym-coexist-01},\ref {lem-asym-coexist-02} and \ref{lem-asym-coexist-03}, we get:
$$
(\overline{u}-\underline{u})\Big(a_1-2\frac{\chi_1}{d_3}k-|\Omega||a_3|\Big)\leq (|a_2|+|\Omega||a_4|+l\frac{\chi_1}{d_3})(\overline{v}-\underline{v}),
$$
and
\begin{equation}\label{Asym-eq-13}
(\overline{v}-\underline{v})\Big(b_2-2\frac{\chi_2}{d_3}l-|\Omega||b_4|\Big)\leq (|b_1|+|\Omega||b_3|+k\frac{\chi_2}{d_3})(\overline{u}-\underline{u}).
\end{equation}
Thus,
$$
(\overline{u}-\underline{u})\Big(a_1-2\frac{\chi_1}{d_3}k-|\Omega||a_3|\Big)\Big(b_2-2\frac{\chi_2}{d_3}l-|\Omega||b_4|d\Big)\leq (|a_2|+|\Omega||a_4|+l\frac{\chi_1}{d_3})(|b_1|+|\Omega||b_3|+k\frac{\chi_2}{d_3})(\overline{u}-\underline{u}).
$$
Last inequality combined with \eqref{Asymp-coexist-eq-4} yield that $\underline{u}=\overline{u}$. Hence, using again inequality \eqref{Asym-eq-13}, we obtain that $\underline{v}=\overline{v}$. Therefore from lemmas  \ref {lem-asym-coexist-01} and \ref{lem-asym-coexist-03} we get
$$
\underline{u}=\overline{u}=\frac{a_0-(a_2+|\Omega|a_4)\overline{v}}{a_1+|\Omega|a_3},
$$
and  from lemmas  \ref {lem-asym-coexist-02} and \ref{lem-asym-coexist-03} we get
$$
\underline{v}=\overline{v}=\frac{b_0-(b_1+|\Omega|b_3)\overline{u}}{b_2+|\Omega|b_4},
$$
and the results follow.
\end{proof}

\subsection{Exclusion}
In this subsection, our aim is to find also conditions on the parameters only which guarantee that  $\underline{u}(u_0,v_0)=\overline{u}(u_0,v_0)=0$ and  $0<\underline{v}(u_0,v_0)=\overline{v}(u_0,v_0)=\frac{b_0}{b_2+|\Omega|b_4}$.  This method is the so called eventual comparison method. 

 Let $u_{0},v_0\in C(\overline{\Omega})$ be given  nonnegative initials such that $0<\|v_0\|_{\infty}$. Since $0<\|v_0\|_{\infty}$, the maximum principle for parabolic equations implies that
 $$0<\|v(\cdot,t;u_0,v_0)\|_{\infty},\quad \forall\ t\geq 0. $$
Next, we prove the following four lemmas which are important steps toward proving the phenomenon of exclusion.
\begin{lemma}
\label{lem-asym-exclusion-01}


Suppose $a_1>\frac{\chi_{1}k}{d_{3}},$ $a_2\geq\frac{\chi_1l}{d_3}$, $a_4\geq 0,$ and  {\bf (H1)}. Then
\begin{equation}\label{Asymp-exclusion-eq-000}
\overline{u}\leq \frac{\left\{ a_0-\left(a_2 +|\Omega|a_4\right)\underline{v}+|\Omega|(a_3)_-\overline{u}\right\}_{+}}{a_{1}-\frac{\chi_{1}k}{d_{3}}}.
\end{equation}
\end{lemma}

\begin{proof}

Using inequality \eqref{Asym-eq-02} and the fact that $a_2\geq \frac{\chi_1l}{d_3}$, we have
\begin{align*}
&u_t-d_1\Delta u+\chi_1 \nabla u \cdot \nabla w\nonumber\\
& \leq u\left\{ a_0-(a_1-\frac{\chi_1 }{d_3}k)u+|\Omega|(a_3)_-\overline{u}-\left(a_2 +|\Omega|a_4\right)\underline{v}\right\} \nonumber\\
&+u\left\{ \left(a_2+|\Omega|((a_3)_-+a_4)+k\frac{\chi_1}{d_3}+2l\frac{\chi_{1}}{d_{3}}\right)\epsilon\right\},
\end{align*}
and thus since $a_1>\frac{\chi_{1}k}{d_{3}},$ \eqref{Asymp-exclusion-eq-000} follows from  similar arguments as of  \eqref{Asymp-coexist-eq-5}  in Lemma \ref{lem-asym-coexist-01}
\end{proof}

\begin{lemma}
\label{lem-asym-exclusion-02}
Suppose  $b_2>\frac{\chi_{2}l}{d_{3}},$  $b_1> \frac{\chi_2k}{d_3},$  and  {\bf (H1)}. Then

\begin{equation}\label{Asymp-exclusion-eq-04}
\overline{v}\leq \frac{\left\{ b_0+|\Omega|(b_3)_-\overline{u}-\left(|\Omega|(b_4)_++\frac{\chi_{2}l}{d_{3}}\right)\underline{v}+|\Omega|(b_4)_-\overline{v}\right\}_{+}}{b_{2}-\frac{\chi_{2}l}{d_{3}}},
\end{equation}
and
\begin{equation}\label{Asymp-exclusion-eq-05}
\underline{v}\geq \frac{\left\{ b_0-(b_1+|\Omega|(b_3)_+)\overline{u}-\left(|\Omega|(b_4)_++\frac{\chi_{2}l}{d_{3}}\right)\overline{v}+|\Omega|(b_4)_-\underline{v}\right\}_{+}}{b_{2}-\frac{\chi_{2}l}{d_{3}}}.
\end{equation}
\end{lemma}

\begin{proof}
Using inequality \eqref{Asym-eq-02} and the fact that $b_1\geq \frac{\chi_2k}{d_3}$, we have :
\begin{align*}
&v_t-d_2\Delta v+\chi_2 \nabla v \cdot \nabla w\nonumber\\
& \leq v\left\{ b_0-(b_2-\frac{\chi_2 }{d_3}k)v+|\Omega|(b_3)_-\overline{u}-\left(|\Omega|(b_4)_++l\frac{\chi_2}{d_3}\right)\underline{v}+|\Omega|(b_4)_-\overline{v}\right\} \nonumber\\
&+v\left\{ \left(|\Omega|((b_3)_-+|b_4|)+l\frac{\chi_2}{d_3}+k\frac{\chi_{1}}{d_{3}}\right)\epsilon\right\},
\end{align*}
and since $b_2> \frac{\chi_2l}{d_3},$  \eqref{Asymp-exclusion-eq-04} follows from  similar arguments as of  \eqref{Asymp-coexist-eq-7}  in Lemma \ref{lem-asym-coexist-02}.

Similarly, we have
\begin{align*}
&v_t-d_2\Delta v+\chi_2 \nabla v \cdot \nabla w\nonumber\\
& \geq v\left\{ b_0-(b_2-\frac{\chi_2 }{d_3}k)v-(b_1+|\Omega|(b_3)_+)\overline{u}-\left(|\Omega|(b_4)_++l\frac{\chi_2}{d_3}\right)\overline{v}+|\Omega|(b_4)_-\underline{v}\right\} \nonumber\\
&-v\left\{ \left(b_1+|\Omega|((b_3)_-+|b_4|)+l\frac{\chi_2}{d_3}+2k\frac{\chi_{1}}{d_{3}}\right)\epsilon\right\},
\end{align*}
and since $b_2> \frac{\chi_2l}{d_3},$  \eqref{Asymp-exclusion-eq-05} follows from  similar arguments as of  \eqref{Asymp-coexist-eq-8}  in Lemma \ref{lem-asym-coexist-02}.
\end{proof}

\begin{lemma}
\label{lem-asym-exclusion-03}
Suppose  $b_2>\frac{\chi_{2}l}{d_{3}},$   $b_1\leq  \frac{\chi_2k}{d_3},$  and  {\bf (H1)}. Then
\begin{equation}\label{Asymp-exclusion-eq-06}
\overline{v}\leq \frac{\left\{ b_0+\left(\frac{\chi_{2}k}{d_{3}}-b_1+|\Omega|(b_3)_-\right)\overline{u}-\left(|\Omega|(b_4)_++\frac{\chi_{2}l}{d_{3}}\right)\underline{v}+|\Omega|(b_4)_-\overline{v}\right\}_{+}}{b_{2}-\frac{\chi_{2}l}{d_{3}}},
\end{equation}
and
\begin{equation}\label{Asymp-exclusion-eq-07}
\underline{v}\geq \frac{\left\{ b_0-(\frac{\chi_{2}k}{d_{3}}+|\Omega|(b_3)_+)\overline{u}-\left(|\Omega|(b_4)_++\frac{\chi_{2}l}{d_{3}}\right)\overline{v}+|\Omega|(b_4)_-\underline{v}\right\}_{+}}{b_{2}-\frac{\chi_{2}l}{d_{3}}}.
\end{equation}
\end{lemma}

\begin{proof}
We have
\begin{align*}
&v_t-d_2\Delta v+\chi_2 \nabla v \cdot \nabla w\nonumber\\
& \leq v\left\{ b_0-(b_2-\frac{\chi_2 }{d_3}k)v+\left(k\frac{\chi_2 }{d_3}-b_1+|\Omega|(b_3)_-\right)\overline{u}-\left(|\Omega|(b_4)_++l\frac{\chi_2}{d_3}\right)\underline{v}\right\} \nonumber\\
&+v\left\{ +|\Omega|(b_4)_-\overline{v}+\left(k\frac{\chi_2 }{d_3}-b_1+|\Omega|((b_3)_-+|b_4|)+l\frac{\chi_2}{d_3}+k\frac{\chi_{1}}{d_{3}}\right)\epsilon\right\},
\end{align*}
and since $b_2> \frac{\chi_2l}{d_3},$  \eqref{Asymp-exclusion-eq-06} follows from  similar arguments as of  \eqref{Asymp-coexist-eq-7}  in Lemma \ref{lem-asym-coexist-02}.

Similarly, we have
\begin{align*}
&v_t-d_2\Delta v+\chi_2 \nabla v \cdot \nabla w\nonumber\\
& \geq v\left\{ b_0-(b_2-\frac{\chi_2 }{d_3}k)v-(\frac{\chi_{2}k}{d_{3}}+|\Omega|(b_3)_+)\overline{u}-\left(|\Omega|(b_4)_++l\frac{\chi_2}{d_3}\right)\overline{v}+|\Omega|(b_4)_-\underline{v}\right\} \nonumber\\
&-v\left\{ \left(|\Omega|((b_3)_++|b_4|)+l\frac{\chi_2}{d_3}+2k\frac{\chi_{1}}{d_{3}}\right)\epsilon\right\},
\end{align*}
and since $b_2> \frac{\chi_2l}{d_3},$  \eqref{Asymp-exclusion-eq-07} follows from  similar arguments as of  \eqref{Asymp-coexist-eq-8}  in Lemma \ref{lem-asym-coexist-02}.
\end{proof}

Now using the previous four lemmas, we prove Theorem \ref{Asym-Th-2}.
\begin{proof}[Proof of Theorem \ref{Asym-Th-2}]
The proof is divided in two steps. In the first step, we prove $\overline{u}=0$ and then in the second step, using the result of first step , we get $\underline{v}=\overline{v}=\frac{b_0}{b_2+|\Omega|b_4}.$

\noindent{\bf Step 1.}  $\overline{u}=0$\\
The proof of $\overline{u}=0$ is also divided into two cases, according to the sign of the quantity $b_1-\frac{\chi_2k}{d_3}$. If $b_1> \frac{\chi_2k}{d_3}$ we shall based our arguments on Lemmas \ref{lem-asym-exclusion-01} and \ref{lem-asym-exclusion-02}. While if $b_1\leq  \frac{\chi_2k}{d_3}$, the arguments of proof are based on Lemmas \ref{lem-asym-exclusion-01} and \ref{lem-asym-exclusion-03}. Since the arguments in both cases are similar, we only provide the proof in case $b_1> \frac{\chi_2k}{d_3}$. Hence, without loss of generality, we might suppose that $b_1> \frac{\chi_2k}{d_3}$.

Suppose by contradiction that $\overline{u}>0.$
First, from equations \eqref{Asymp-exclusion-eq-000}, \eqref{Asymp-exclusion-eq-01} and $\overline{u}>0,$ we get
\begin{equation}\label{Asymp-exclusion-eq-08}
\underline{v}<\frac{a_0}{a_2+|\Omega|a_{4}}
\end{equation}


In this case, from \eqref{Asymp-exclusion-eq-08} and \eqref{Asymp-exclusion-eq-02} we get
\begin{align*}
b_0-\left(|\Omega|(b_4)_++\frac{\chi_{2}l}{d_{3}}\right)\underline{v}+|\Omega|(b_4)_-\overline{v}&\geq b_0-\left(|\Omega|b_4+\frac{\chi_{2}l}{d_{3}}\right)\underline{v}\nonumber\\
&> b_0-\frac{|\Omega|b_4+\frac{\chi_{2}l}{d_{3}}}{ a_2+|\Omega|a_{4}}a_0\nonumber\\
&\geq  b_0-\frac{|\Omega|b_4+b_2}{ a_2+|\Omega|a_{4}}a_0\nonumber\\
&\geq 0.
\end{align*}
Therefore
\begin{equation}\label{Asymp-exclusion-eq-09}
b_0+|\Omega|(b_3)_-\overline{u}-\left(|\Omega|(b_4)_++\frac{\chi_{2}l}{d_{3}}\right)\underline{v}+|\Omega|(b_4)_-\overline{v}>0.
\end{equation}
From \eqref{Asymp-exclusion-eq-05}, we get
\begin{equation*}
    ((b_{4})_+|\Omega|+\frac{\chi_2}{d_3}l)\overline{v}\geq b_0- (b_{1}+(b_{3})_+|\Omega|)\overline{u}-(b_{2}-\frac{\chi_{2}}{d_{3}}l-|\Omega|(b_4)_-)\underline{v}.
\end{equation*}
Thus, from equations  \eqref{Asymp-exclusion-eq-000}, \eqref{Asymp-exclusion-eq-01} and $\overline{u}>0,$ we get
\begin{equation*}
    ((b_{4})_+|\Omega|+\frac{\chi_2}{d_3}l)\overline{v}\geq b_0- (b_{1}+(b_{3})_+|\Omega|)\frac{\left\{  a_0-(a_2+a_{4}|\Omega|)\underline{v}\right\}}{a_{1}-\frac{\chi_{1}k}{d_{3}}-|\Omega|(a_3)_-}-(b_{2}-\frac{\chi_{2}}{d_{3}}l-|\Omega|(b_4)_-)\underline{v}.
\end{equation*}
Therefore
\begin{align*}
    &((b_{4})_+|\Omega|+\frac{\chi_2}{d_3}l)\overline{v}\nonumber\\
&\geq  b_0- \frac{b_{1}+(b_{3})_+|\Omega|}{a_{1}-\frac{\chi_{1}k}{d_{3}}-|\Omega|(a_3)_-}a_0\nonumber\\
&-\frac{(b_{2}-\frac{\chi_{2}}{d_{3}}l-|\Omega|(b_4)_-)(a_{1}-\frac{\chi_{1}k}{d_{3}}-|\Omega|(a_3)_-)-(b_{1}+(b_{3})_+|\Omega|)(a_2+a_{4}|\Omega|)}{a_{1}-\frac{\chi_{1}k}{d_{3}}-|\Omega|(a_3)_-}\underline{v}.
\end{align*}
It follows from the last inequality and inequality \eqref{Asymp-exclusion-eq-04} that
\begin{align*}
&((b_{4})_+|\Omega|+\frac{\chi_2}{d_3}l)\frac{  b_0+|\Omega|(b_3)_-\overline{u}-\left(|\Omega|(b_4)_++\frac{\chi_{2}l}{d_{3}}\right)\underline{v}}{b_{2}-\frac{\chi_{2}}{d_{3}}l-|\Omega|(b_4)_-}\nonumber\\
&\geq b_0- \frac{(b_{1}+(b_{3})_+|\Omega|)}{a_{1}-\frac{\chi_{1}k}{d_{3}}-|\Omega|(a_3)_-}a_0\nonumber\\
&-\frac{(b_{2}-\frac{\chi_{2}}{d_{3}}l-{|\Omega|(b_4)_-})(a_{1}-\frac{\chi_{1}k}{d_{3}}-|\Omega|(a_3)_-)-(b_{1}+(b_{3})_+|\Omega|)(a_2+a_{4}|\Omega|)}{a_{1}-\frac{\chi_{1}k}{d_{3}}-|\Omega|(a_3)_-}\underline{v}.
\end{align*}
Thus
\begin{align*}
& \frac{(b_{2}-\frac{\chi_{2}}{d_{3}}l-|\Omega|(b_4)_-)(a_{1}-\frac{\chi_{1}k}{d_{3}}-|\Omega|(a_3)_-)-(b_{1}+(b_{3})_+|\Omega|)(a_2+a_{4}|\Omega|)}{a_{1}-\frac{\chi_{1}k}{d_{3}}-|\Omega|(a_3)_-}\underline{v}\nonumber\\
&-\frac{  ((b_{4})_+|\Omega|+\frac{\chi_2}{d_3}l)^2 }{b_{2}-\frac{\chi_{2}}{d_{3}}l-|\Omega|(b_4)_-}\underline{v}\nonumber\\
&\geq\left\{1-\frac{((b_{4})_+|\Omega|+\frac{\chi_2}{d_3}l)}{b_{2}-\frac{\chi_{2}}{d_{3}}l-|\Omega|(b_4)_-}\right\} b_0- \frac{(b_{1}+(b_{3})_+|\Omega|)a_0}{a_{1}-\frac{\chi_{1}k}{d_{3}}-|\Omega|(a_3)_-}-\frac{\left((b_{4})_+|\Omega|+\frac{\chi_2}{d_3}l\right)|\Omega|(b_{3})_-}{b_{2}-\frac{\chi_{2}}{d_{3}}l-|\Omega|(b_4)_-}\overline{u}.
\end{align*}
Using equations  \eqref{Asymp-exclusion-eq-000}, \eqref{Asymp-exclusion-eq-01} and $\overline{u}>0,$ it follows from the last inequality that
\begin{align*}
& \frac{(b_{2}-\frac{\chi_{2}}{d_{3}}l-|\Omega|(b_4)_-)(a_{1}-\frac{\chi_{1}k}{d_{3}}-|\Omega|(a_3)_-)-(b_{1}+(b_{3})_+|\Omega|)(a_2+a_{4}|\Omega|)}{a_{1}-\frac{\chi_{1}k}{d_{3}}-|\Omega|(a_3)_-}\underline{v}\nonumber\\
&-\frac{  ((b_{4})_+|\Omega|+\frac{\chi_2}{d_3}l)^2 }{b_{2}-\frac{\chi_{2}}{d_{3}}l-|\Omega|(b_4)_-}\underline{v}\nonumber\\
&\geq\left\{1-\frac{((b_{4})_+|\Omega|+\frac{\chi_2}{d_3}l)}{b_{2}-\frac{\chi_{2}}{d_{3}}l-|\Omega|(b_4)_-}\right\} b_0- \frac{(b_{1}+(b_{3})_+|\Omega|)}{a_{1}-\frac{\chi_{1}k}{d_{3}}-|\Omega|(a_3)_-}a_0\nonumber\\
&-{\frac{\left(|\Omega|(b_{4})_++\frac{\chi_2}{d_3}l\right)\big(  a_0-(a_2+a_{4}|\Omega|)\underline{v}\big)|\Omega|(b_{3})_-}{\left(b_{2}-\frac{\chi_{2}}{d_{3}}l-|\Omega|(b_4)_-\right)\left(a_{1}-\frac{\chi_{1}k}{d_{3}}-|\Omega|(a_3)_-\right)}.}
\end{align*}
Thus, we get
\begin{align}\label{ra-00001}
&\underbrace{\left\{ (a_{1}-\frac{\chi_{1}k}{d_{3}}-|\Omega|(a_3)_-)(b_{2}-2\frac{\chi_{2}}{d_{3}}l-|b_{4}||\Omega|)(b_2+b_{4}|\Omega|)\right\}}_{B_1}\underline{v}\nonumber\\
&-\underbrace{\left\{ (b_{1}+(b_{3})_+|\Omega|)(a_2+a_{4}|\Omega|)(b_{2}-\frac{\chi_{2}}{d_{3}}l-|\Omega|(b_4)_-)\right\}}_{B_2}\underline{v}\nonumber\\
&\geq\underbrace{(a_{1}-\frac{\chi_{1}k}{d_{3}}-|\Omega|(a_3)_-)(b_{2}-2\frac{\chi_{2}}{d_{3}}l-|b_{4}||\Omega|) b_0- (b_{2}-\frac{\chi_{2}}{d_{3}}l-|\Omega|(b_4)_-)(b_{1}+(b_{3})_+|\Omega|)a_0}_{A_1}\nonumber\\
&-\underbrace{((b_{4})_+|\Omega|+\frac{\chi_2}{d_3}l)a_0|\Omega|(b_{3})_-}_{A_2}+\underbrace{\left\{(|\Omega|(b_{4})_+ +\frac{\chi_2}{d_3}l)(a_2+a_{4}|\Omega|)|\Omega|(b_{3})_-\right\}}_{B_3}\underline{v}.
\end{align}

Then, inequality \eqref{ra-00001} is equivalent to
\begin{equation}\label{ra-aa1}
B\underline{v}\geq A
\end{equation}
where $B=B_1- B_2 - B_3$ and $A=A_1 - A_2.$
Note that the first equation of \eqref{Asymp-exclusion-eq-03} yields that   $A>0.$ This combined with \eqref{ra-aa1} imply that   $B>0$.  Therefore, inequality \eqref{ra-aa1} becomes
$$
\underline{v}\geq \frac{A}{B}.
$$
Then thanks  to equation \eqref{Asymp-exclusion-eq-08}, we get
$$
B>\frac{a_2+|\Omega|a_4}{a_0}A.
$$
That means
\begin{align*}
 &(a_{1}-\frac{\chi_{1}k}{d_{3}}-|\Omega|(a_3)_-)(b_{2}-2\frac{\chi_{2}}{d_{3}}l-|b_{4}||\Omega|)(b_2+b_{4}|\Omega|)\nonumber\\
&>(a_{1}-\frac{\chi_{1}k}{d_{3}}-|\Omega|(a_3)_-)(b_{2}-2\frac{\chi_{2}}{d_{3}}l-|b_{4}||\Omega|)\frac{a_2+|\Omega|a_4}{a_0}b_0,
\end{align*}
Thus
$$
(a_{1}-\frac{\chi_{1}k}{d_{3}}-|\Omega|(a_3)_-)(b_{2}-2\frac{\chi_{2}}{d_{3}}l-|b_{4}||\Omega|)\left\{\frac{a_2+|\Omega|a_4}{a_0}b_0-(b_2+b_{4}|\Omega|)\right\}<0.
$$
which contradict equations \eqref{Asymp-exclusion-eq-01}, \eqref{Asymp-coexist-eq-02} and  \eqref{Asymp-exclusion-eq-02} .

\medskip

\medskip
\noindent{\bf Step 2.} Since by {\bf Step 1.} $\overline{u}=0,$ we get that   \eqref{Asymp-exclusion-eq-04} and \eqref{Asymp-exclusion-eq-06} are equivalent and become
\begin{equation}\label{Asymp-exclusion-eq-010}
\overline{v}\leq \frac{ b_0-\left(|\Omega|(b_4)_++\frac{\chi_{2}l}{d_{3}}\right)\underline{v}+|\Omega|(b_4)_-\overline{v}}{b_{2}-\frac{\chi_{2}l}{d_{3}}}.
\end{equation}
Similarly, we get that \eqref{Asymp-exclusion-eq-05} and \eqref{Asymp-exclusion-eq-07} are equivalent and become
\begin{equation}\label{Asymp-exclusion-eq-011}
\underline{v}\geq \frac{ b_0-\left(|\Omega|(b_4)_++\frac{\chi_{2}l}{d_{3}}\right)\overline{v}+|\Omega|(b_4)_-\underline{v}}{b_{2}-\frac{\chi_{2}l}{d_{3}}}.
\end{equation}
By taking the difference \eqref{Asymp-exclusion-eq-010}-\eqref{Asymp-exclusion-eq-011}, we get
$$
\left(b_{2}-2\frac{\chi_{2}}{d_{3}}l-|b_{4}||\Omega|\right)\left(\overline{v}-\underline{v}\right)\leq 0.
$$
Thus  by \eqref{Asymp-coexist-eq-02} we get $\overline{v}=\underline{v}$ and it then follows from \eqref{Asymp-exclusion-eq-010} and \eqref{Asymp-exclusion-eq-011} that
$$
\overline{v}=\underline{v}=\frac{b_0}{b_2+|\Omega|b_4}.
$$
\end{proof}

\medskip

\noindent {\bf Perspectives.} This study showed that even in the case of parabolic-parabolic-elliptic chemotaxis system with Lotka-Volterra type sources and nonlocals competitive terms, the eventual comparison method gives explicit natural parameter regions for both coexistence and exclusion phenomenons. A natural and non trivial question is wether the  method of eventual comparison can be entended to the study  of the full parabolic chemotaxis system of two species and one chemoattractants that is, 
\begin{equation}\label{perspective-eq1}
\begin{cases}
u_t=d_1\Delta u-\chi_1\nabla (u\cdot \nabla w)+u\left(a_0-a_1u-a_2v\right),\quad x\in \Omega \quad\cr
v_t=d_2\Delta v-\chi_2\nabla (v\cdot \nabla w)+v\left(b_0-b_1u-b_2v\right),\quad x\in \Omega \quad\cr
\tau w_t=d_3\Delta w+k u+lv-\lambda w,\quad x\in \Omega, \,\, \tau > 0 \quad\cr
\end{cases}
 \end{equation}
with homogeneous Neuman boundary conditions on bounded (convex) domains. The main challenge here is that $\Delta w$ or equivalently $w_t$  may not be small when $u$ and $v$ are small.  But for the method to work, we  only need  to find (explicit) bound for $\|\Delta w\|_\infty$ for time large enougth. Note that this question remains open even in the case of one species full parabolic of chemotaxis system.  An execellent reference in this direction is the recent paper of Winkler  \cite{W2014a}, where the author got a natural non explcit condition for the asymtotic stability of the constant steady state in one species  full parabolic chemotaxis system on bounded convex domains.

Another interesting challenge is to develop new techniques which would provide explicit and natural hypothesises on the parameters regions  for the study of the asymptotic behavior of system \eqref{perspective-eq1} in heterogeneous medium,
  
\begin{equation}\label{perspective-eq1-heterogeneous}
\begin{cases}
u_t=d_1\Delta u-\chi_1\nabla (u\cdot \nabla w)+u\left(a_0(x,t)-a_1(t,x)u-a_2(x,t)v\right),\quad x\in \Omega \quad\cr
v_t=d_2\Delta v-\chi_2\nabla (v\cdot \nabla w)+v\left(b_0(x,t)-b_1(x,t)u-b_2(x,t)v\right),\quad x\in \Omega \quad\cr
\tau w_t=d_3\Delta w+k u+lv-\lambda w,\quad x\in \Omega, \,\, \tau \geq 0 
\end{cases}
\end{equation}
with homogeneous Neuman boundary conditions on bounded domains. One particular challenge in this case is the existence and nonlinear stability  of positive entire solutions.  We refer to the paper of Issa and Shen, \cite{ITBWS16}, for some existing works in this direction. Finally, it is also very interesting to study the existence of travelling waves for sytem \eqref{perspective-eq1-heterogeneous}.  See the paper of Salako and Shen \cite{RBSWS17b} for the case of constant coefficients.

\medskip

\noindent {\bf Acknowledgment.} We express our sincere gratitude and appreiciation to Professors  Wenxian Shen, J. Ignacio Tello, Michael Winkler and Johannes Lankeit for their valuable discussions, suggestions,  and references.  We also thank the refree for his/her valuable comments and suggestions which greatly improved the  paper, its presentation and style.

\end{document}